\documentclass[final]{svjour3}
\usepackage{xcolor}
\usepackage{graphicx}
\usepackage{mathptmx}      % use Times fonts if available on your TeX system
\usepackage{latexsym}
\journalname{Numer.\ Algorithms}
\usepackage{mathptmx}
\usepackage{amssymb,amsmath,amsfonts}
\usepackage{bm}
\usepackage[english]{babel}
\usepackage{babelbib}
\usepackage[T1]{fontenc}
\usepackage{verbatim}
\usepackage{xcolor}
\usepackage{verbatim}
\usepackage{mathrsfs}
\usepackage{placeins}
\usepackage{bigints}
\usepackage{longtable}
\usepackage{subfigure}
\usepackage{fancyhdr}
\usepackage{dsfont}
\usepackage{calrsfs}
\DeclareMathAlphabet{\pazocal}{OMS}{zplm}{m}{n}
\numberwithin{equation}{section}
\newcommand\scheme[1]{\mbox{\texttt{#1}}}

\newcommand\nl{\\[0.75\jot]} % \\ with some extra space
\newcommand{\CC}{{{\mathbb{C}}\vphantom{|}}}
\newcommand{\NN}{{{\mathbb{N}}\vphantom{|}}}

\newcommand{\dd}{{\mathrm{d}}}
\newcommand{\ee}{{\mathrm{e}}}
\newcommand{\ii}{{\mathrm{i}}}

\newcommand{\nice}[1]{\pazocal{#1}} %\mathscr

\newcommand{\nL}{\nice{L}}

\newcommand{\nP}{\nice{P}}

\newcommand{\nS}{\nice{S}}

\renewcommand\,{\ensuremath{\mspace{1.5mu}}} %%%%%%%%%%%%%%%%%%
\newcommand\Order{{\mathscr{O}}}
\newcommand{\norm}[2]{\|#2\|_{#1}}
\newcommand{\normbig}[2]{\big\|#2\big\|_{#1}}
\newcommand\pat{\partial_{\,t}}

\renewcommand\d{\cdot}

\newcommand{\ZZ}{{{\mathbb{Z}}\vphantom{|}}}
\newcommand\il{\langle}
\newcommand\ir{\rangle}

\newenvironment{rev}{}{}%

\begin{document}
%%%%%%%%%%%%%%%%%%%%%%%%%%%%%%%%%%%%%%%%%%%%%%%%%%%%%%%%%%%%%%%%%%
%%%%%%%%%%%%%%%%%%%%%%%%%%%%%%%%%%%%%%%%%%%%%%%%%%%%%%%%%%%%%%%%%%
\title{Adaptive high-order splitting methods for systems of nonlinear evolution equations with periodic boundary conditions
\thanks{This work was supported by the Austrian Science Fund (FWF) under grant P24157-N13 and the Vienna Science and Technology Fund (WWTF) under the grant MA14-002. The computational results presented have been achieved in
part using the Vienna Scientific Cluster (VSC). We would like to thank Benson K. Muite (University of Tartu, Estonia)
for many helpful discussions and implementation of a first version of our numerical solver.}}

%\subtitle{The and Strang splitting methods for nonlinear Schr\"{o}dinger equations in the semiclassical regime}

\titlerunning{Adaptive high-order splitting methods for nonlinear evolution equations}        % if too long for running head

\author{Winfried Auzinger, Othmar Koch, and Michael Quell}

%\authorrunning{Short form of author list} % if too long for running head

\institute{Winfried Auzinger \at
              Institut f{\"u}r Analysis und Scientific Computing,
              Technische Universit{\"a}t Wien,
              Wiedner Hauptstra{\ss}e \mbox{8-10/E101},
              A-1040 Wien, Austria \\
              \email{w.auzinger@tuwien.ac.at}           %  \\
%             \emph{Present address:} of F. Author  %  if needed
           \and
           Othmar Koch \at
              Fakult{\"a}t f{\"u}r Mathematik,
              Universit{\"a}t Wien,
              Oskar-Morgenstern-Platz~1, A-1090 Wien, Austria \\
              \email{othmar@othmar-koch.org}
           \and
           Michael Quell \at
              Institut f{\"u}r Analysis und Scientific Computing,
              Technische Universit{\"a}t Wien,
              Wiedner Hauptstra{\ss}e \mbox{8-10/E101},
              A-1040 Wien, Austria \\
              \email{michael.quell@yahoo.de}           %  \\
}

\date{Received: \today / Accepted: date}
% The correct dates will be entered by the editor
\maketitle

\begin{abstract}
We assess the applicability and efficiency of time-adaptive high-order splitting methods
applied for the numerical solution of (systems of) nonlinear parabolic problems under periodic
boundary conditions. We discuss in particular several applications generating intricate patterns
and displaying nonsmooth solution dynamics. First we give a general error analysis for splitting
methods for parabolic problems under periodic boundary conditions and derive the necessary
smoothness requirements on the exact solution in particular for the Gray--Scott equation
\begin{rev}and the Van der Pol equation.\end{rev}
Numerical examples demonstrate the convergence of the methods and serve to compare
the efficiency of different time-adaptive splitting schemes
and of splitting into either two or three operators, based on appropriately constructed
a~posteriori local error estimators.
% \PACS{PACS code1 \and PACS code2 \and more}
\end{abstract}

\section{Introduction}\label{sec:intro}

We are interested in computational methods for nonlinear evolution
equations of the type
\begin{equation}\label{de1}
\partial_t u(t) =  A\,u(t) + B(u(t)),\quad t>t_0,
\end{equation}
on a Banach space $\nice{B}$,
\begin{rev}which in our examples equals $L^2$ on the $d$-dimensional torus\end{rev}.
Here, $A: \nice{D}\subseteq \nice{B} \to \nice{B}$ is an (unbounded)
differential operator and $B$ a generally unbounded nonlinear operator
\begin{rev}whose domain has nonempty intersection with $\nice{D}$\end{rev}.

To enable an efficient numerical solution of~\eqref{de1}
for large-scale applications, adaptive high-order time-discretizations are central. In some applications
the promised speed-up will be critical for the feasibility of a simulation.
In many realistic models, the stiffness of the operators $A$ and $B$ is different which suggests
to use splitting methods which separately propagate the two vector fields. If $A$ is a linear differential operator, effective schemes
are known which solve the subproblem efficiently after appropriate space discretization.
For the problems discussed in this paper, a Fourier pseudospectral space discretization
is the most natural choice as this allows to propagate the linear part by exponentiation of
a diagonal matrix.

Parabolic equations often induce high computational demand due to challenging solution dynamics,
which suggests to employ adaptive time-stepping
\begin{rev}in order to accommodate for local variations in the numerical error.
However, this is not the only reason for using adaptivity. Typically, the optimal step-size
is not known a priori, and an adaptive procedure determines the appropriate value
within a few steps, see for example Section~\ref{subsec:compare_gs}. Moreover, adaptive time-stepping
increases the reliability of a computation, see for instance~\cite{soederlind06b}.
\end{rev}

At the (time-)semi-discrete level, $s$-stage exponential splitting methods for the integration of~\eqref{de1}
use multiplicative combinations of the partial flows $ \phi_A(t,u) $ and $ \phi_B(t,u) $.
For a single step $ (0,u_0) \mapsto (h,u_1) $ with time-step $ t=h $, this reads
\begin{equation}\label{splitting1}
u_{1} := \nice{S}(h,u_0) = \phi_B(b_s h,\cdot) \circ \phi_A(a_s h,\cdot) \circ \ldots \circ
                              \phi_B(b_1 h,\cdot) \circ \phi_A(a_1 h,u_0),
\end{equation}
where the coefficients $ a_j,b_j,\, j=1 \ldots s $ are determined according to the requirement
that a prescribed order of consistency is obtained~\cite{haireretal02}.

Compared to highly implicit methods as for instance implicit Runge--Kutta methods
or their exponential counterparts (see~\cite{hocost10}), splitting methods are easy to implement
and efficient in combination with suitable spatial discretization and appropriate implementations or approximations
of the subflows $ \phi_A $ and $ \phi_B $.
\begin{rev}
This is an important asset of our approach,
however we will demonstrate in addition that adaptive choice of the time steps
leads to a more efficient solution for problems where the variation
in the solution is large.
\end{rev}
For related work on adaptivity using a pair of lower order
methods we refer to~\cite{descombesetal11}.

A rigorous error analysis of splitting methods for Schr\"odinger equations has first been given
for the second-order Strang splitting scheme in~\cite{jahlub00}, which has later been extended
to higher-order splittings in~\cite{thalhammer08}. The more involved arguments for the
nonlinear case have been devised in~\cite{lubich07} for the Schr\"odinger--Poisson
and cubic nonlinear Schr\"odinger equation for second order splitting; higher-order
methods are analyzed in~\cite{knth10a}.

The error analysis relies on an error representation which was first proven in
\cite{knth10a}: the local error of a splitting
method of order $p$ applied to a nonlinear evolution equation has
an error expansion with leading term
\begin{equation}\label{errorexpansion}
\nL(h,u)\sim\sum_{k=1}^p \!
\underset{|{\mu}| \leq p-k}{\sum_{\mu \in \NN^k}} \tfrac{1}{\mu!}
h^{k+|\mu|} C_{k \mu} \prod_{\ell = 1}^{k}
\mathrm{ad}_{D_A}^{\mu_\ell}(D_B) \ee^{h D_A} u,
\end{equation}
where $C_{k\mu}$ are computable constants and $D_A,\ D_B$ represent the Lie derivatives
of the two vector fields, respectively.
$\mathrm{ad}_{D_A}^{\mu_\ell}(D_B)$ denotes the $\mu_\ell$-fold commutator.
In our subsequent analysis we will
make use of this error representation, where the main task will be to
compute and estimate the commutators of the vector fields in an appropriate
functional analytic setting in the space of periodic functions.
To this end, we will resort to a Sobolev theory on the torus,
which we review in detail in Appendix~\ref{sec:period}, to which we refer for notations
used in the subsequent error analysis.

Detailed understanding and analysis of splitting methods for parabolic
problems in particular for the nonlinear case is missing to date.
Partial results have been obtained by other authors;
recent work for linear problems can be found
in~\cite{blanesetal13} and~\cite{hansenetal09a}. In particular, in~\cite{blanesetal13},
a number of higher order methods with complex coefficients are constructed. In these papers,
splitting methods are analyzed in the context of semigroup theory.
However, the authors do not exploit the special structure of the local error
(as specified in~\cite{auzingeretal13a} in terms of iterated commutators).
Therefore the results in~\cite{hansenetal09a} rely on unnaturally restrictive regularity
assumptions, and the same is true for the convergence results given in~\cite{blanesetal13}.

Section~\ref{sec:errest} introduces a number of local and global a~posteriori error
estimators whose performance will subsequently be assessed.

In Section~\ref{sec:gs}, our theoretical framework is applied to analyze the convergence
of splitting methods for the Gray--Scott equation, where the regularity
requirements on the exact solution are worked out which ensure boundedness
of the commutators appearing in the error expansion.

{{In Section~\ref{sec:VP}, we investigate the
Van der Pol system, which has a stiff limit cycle. Adaptive time-stepping is shown to give
rise to guaranteed accuracy, and in some cases significantly reduced computation times compared to fixed time steps.}}

In Section~\ref{sec:abc} we demonstrate that splitting into three operators
can be beneficial computationally if the structure of the vector field
enables exact integration of the subproblems, by resorting to computations
for the Gray--Scott equations.

The functional analytic framework for the error analysis of splitting methods
applied to parabolic problems under periodic boundary conditions is briefly recapitulated in Appendix~\ref{sec:period},
which states the underlying results for the space of periodic functions on the torus.
Sobolev embeddings which are used in our error estimates are stated in Appendix~\ref{sec:sobemb}
with a brief indication of the proofs.

\section{A posteriori local error estimators} \label{sec:errest}

In this section, we briefly describe \begin{rev}three\end{rev} classes of computable a~posteriori local error estimators
which serve as our basis for adaptive time-stepping and which have different
advantages depending on the context in which they are applied.
\emph{Embedded pairs of splitting formulae} have been introduced in
\cite{knth10b} and are based on reusing a number of evaluations from
the basic integrator.
\begin{rev}For methods of odd order\end{rev}, an asymptotically correct error
estimator can be computed at the same cost as for the basic method
by employing the adjoint method, see~\cite{part1},
and finally the \emph{Milne device} relies on the explicit knowledge of the
leading error terms of methods of equal order. A collection of splitting coefficients
covering also these three types of error estimators has been compiled at the webpage
\begin{center} \texttt{http://www.asc.tuwien.ac.at/\~{}winfried/splitting/} \end{center}
which we subsequently refer to as~\cite{splithp}.

\subsection{Embedded pairs} \label{subsec:embed}

In~\cite{knth10b}, pairs of splitting schemes of orders $ p $ and $ p+1 $ are specified.
The idea is to select a controller $ \bar \nS $ of order $ p+1 $ and to construct \begin{rev}an integrator\end{rev} $ \nS $
of order~$ p $ for which a maximal number of compositions coincide with those of the controller.
To construct pairs offering an optimal balance between cost and accuracy, we fix a `good' controller of order $ p+1 $
and wish to adjoin to it a `good' integrator of order $ p $.
Since the number of compositions $ \bar{s} $ in the controller
will be higher than the number of compositions $ s $ in the integrator,
we can select an optimal embedded integrator
$ \nS $ from a set of candidates obtained by flexible embedding,
where the number of coinciding coefficients is not a priori fixed. The idea is expanded in detail in
\cite{part1}, where optimized methods are determined.

\subsection{Adjoint pairs and palindromic formulae} \label{subsec:palindromic}

\begin{rev}For a scheme $ \nS $ of odd order $ p $\end{rev}, the leading local error terms of $ \nS $ and its adjoint $ \nS^\ast $
are identical up to the factor $ -1 $, see \cite{haireretal02}. Therefore, the averaged additive scheme
\begin{equation}
\label{eq:pali-avaraged}
{\bar\nS}(h,u) = \tfrac{1}{2}\,\big( \nS(h,u)+{\nice{S}}^\ast(h,u) \big)
\end{equation}
is a method of order $ p+1 $,
%\footnote{For the simplest case of the Lie-Trotter scheme this has already been observed in~\cite{strang68}.}
and
\begin{equation*}
\nP(h,u) := \nS(h,u) - {\bar\nS}(h,u) = \tfrac{1}{2}\,\big( \nS(h,u) - {\nice{S}}^\ast(h,u) \big)
\end{equation*}
provides an asymptotically correct local error estimate for $ \nS(h,u) $.
In this case the additional effort for computing the local error estimate is
identical with the effort for the integrator $ \nS $ but not higher as is the case
for embedded pairs.
\begin{rev}
This principle is limited to methods of odd order.
In particular, in \cite{part1} so-called palindromic schemes were constructed
which turn out to have small error constants as compared to competing schemes.
Therefore, we include palindromic pairs in our investigations.
\end{rev}

\subsection{The Milne device} \label{subsec:milne}

In the context of multi-step methods for ODEs,
the so-called Milne device is a well-established technique for constructing
pairs of schemes. In our context, one may aim for finding a pair
$ (\nS,\tilde \nS) $ of schemes of equal order $ p $ such that their local errors
$ \nL,\tilde \nL $ are related according to
\begin{subequations} \label{eq:le-milne}
\begin{align}
\nL(h,u)  &= ~~C(u)\,h^{p+1} + \Order(h^{p+2}), \label{eq:le-milne1} \\
\tilde \nL(h,u) &= \gamma\,\,C(u)\,h^{p+1} + \Order(h^{p+2}), \label{eq:le-milne2}
\end{align}
\end{subequations}
with $ \gamma \not= 1 $. Then, the additive scheme
\begin{equation*}
{\bar\nS}(h,u) = -\tfrac{\gamma}{1-\gamma}\,\nS(h,u) + \tfrac{1}{1-\gamma}\,\tilde \nS(h,u)
\end{equation*}
is a method of order $ p+1 $, and
\begin{equation*}
\nP(h,u) := \nS(h,u) - {\bar\nS}(h,u) = \tfrac{1}{1-\gamma}\,\big( \nS(h,u) - \tilde\nS(h,u) \big)
\end{equation*}
provides an asymptotically correct local error estimate for $ \nS(h,u) $.

\subsection{Step-size selection}\label{subsec:adaptstep}

Based on a local error estimator, the step-size is adapted such that a prescribed local error tolerance
\textrm{tol} is expected to be satisfied in the subsequent step. If $h_{\text{old}}$ denotes the
current step-size, the next step-size $h_{\text{new}}$ is predicted as (see~\cite{haireretal87,recipes88})
\begin{equation}
\label{step-selct}
h_{\text{new}} = h \cdot \min\Big\{\alpha_{\text{max}},\max\Big\{\alpha_{\text{min}},
\Big(\alpha\,\dfrac{\textrm{tol}}{\nP(h_{\text{old}})}\Big)^{\frac{1}{p+1}}\,\Big\}\Big\},
\end{equation}
where we choose \begin{rev}$ \alpha = 0.9$\end{rev}, $ \alpha_{\text{min}} = 0.25 $, $ \alpha_{\text{max}} = 4.0 $.
This simple strategy incorporates safety factors to avoid an oscillating and unstable behavior.
\begin{rev}
The chosen values of $ \alpha_{\text{min}}$ and $ \alpha_{\text{max}}$
are commensurable with the recommendations in~\cite{haireretal02}.
The safety factors have not proven critical in our examples, the local changes
in the stepsizes are usually smaller from step to step, see for example
Figure~\ref{fig:change_vdp}. Only if at
the beginning of time propagation the initial stepsize is unsuitable as
in Figure~\ref{fig:change_gs}, where still no instabilities arise in
the step-size control, however.
\end{rev}
%More elaborate strategies based on digital filters and control theory are put forward
%in~\cite{soederlind01,soederlind03,soederlind06a,soederlind06b}, these will not be
%investigated in this work, however.

\section{The Gray--Scott equation}\label{sec:gs}

As a concrete example, we first study the Gray-Scott system (see ~\cite{grayscott90}) modeling a two-component
reaction-diffusion process,
\begin{subequations}\label{grayscott}
\begin{eqnarray}
&& \partial_t u(x,y,t) = c_u\,\Delta u(x,y,t) - u(x,y,t)\,v^2(x,y,t) + \alpha(1-u(x,y,t)), \\[1\jot]
&& \partial_t v(x,y,t) = c_v\,\Delta v(x,y,t) + u(x,y,t)\,v^2(x,y,t) - \beta\,v(x,y,t).
\end{eqnarray}
\end{subequations}
This system is of the type~\eqref{de1}, with unknown $ (u(x,y,t),v(x,y,t)) $, the vector of concentrations
of the two chemical species involved. In many situations this model is closed
naturally by periodic boundary conditions.
This system is studied as a model for pattern formation with a rich dynamical behavior.
For $(x,y)\in [-4\pi,4\pi]^2$ we prescribe the initial condition
\begin{equation}\label{inidata1}
u(x,y,0)=0.5+\exp(-1-(x^2+y^2)),\quad
v(x,y,0)=0.1+\exp(-1-(x^2+y^2)).
\end{equation}
A visualization of the solution component $v$ at $t=0,\ 200\begin{rev}0\end{rev}$  and
$400\begin{rev}0\end{rev}$ is shown in Figure~\ref{fig:gs_sol}.

\begin{figure}[h!]
\begin{center}
\includegraphics[width=3.5cm]{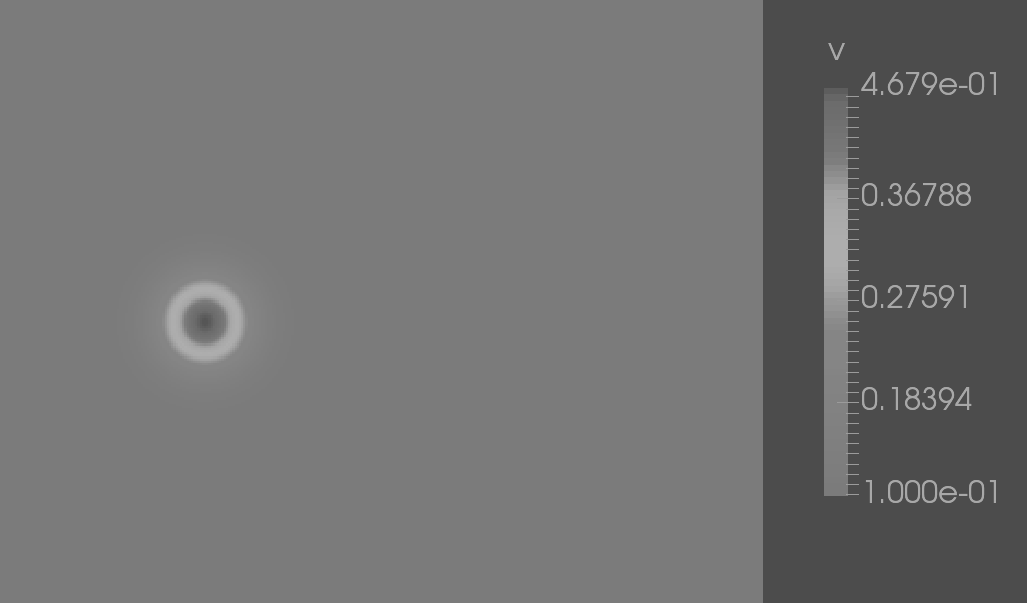} \hspace{5mm}
\includegraphics[width=3.6cm]{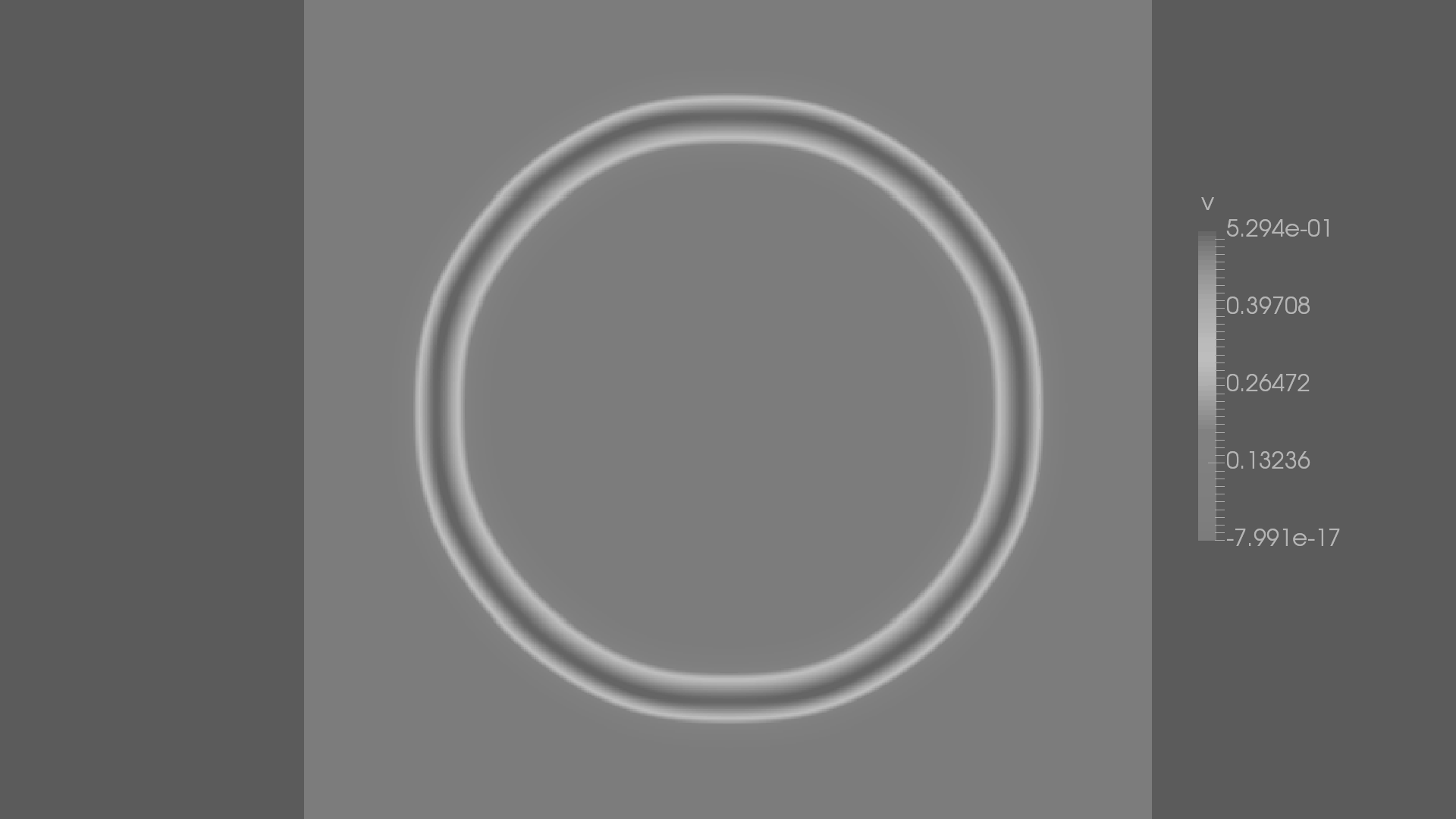} \hspace{5mm}
\includegraphics[width=3.6cm]{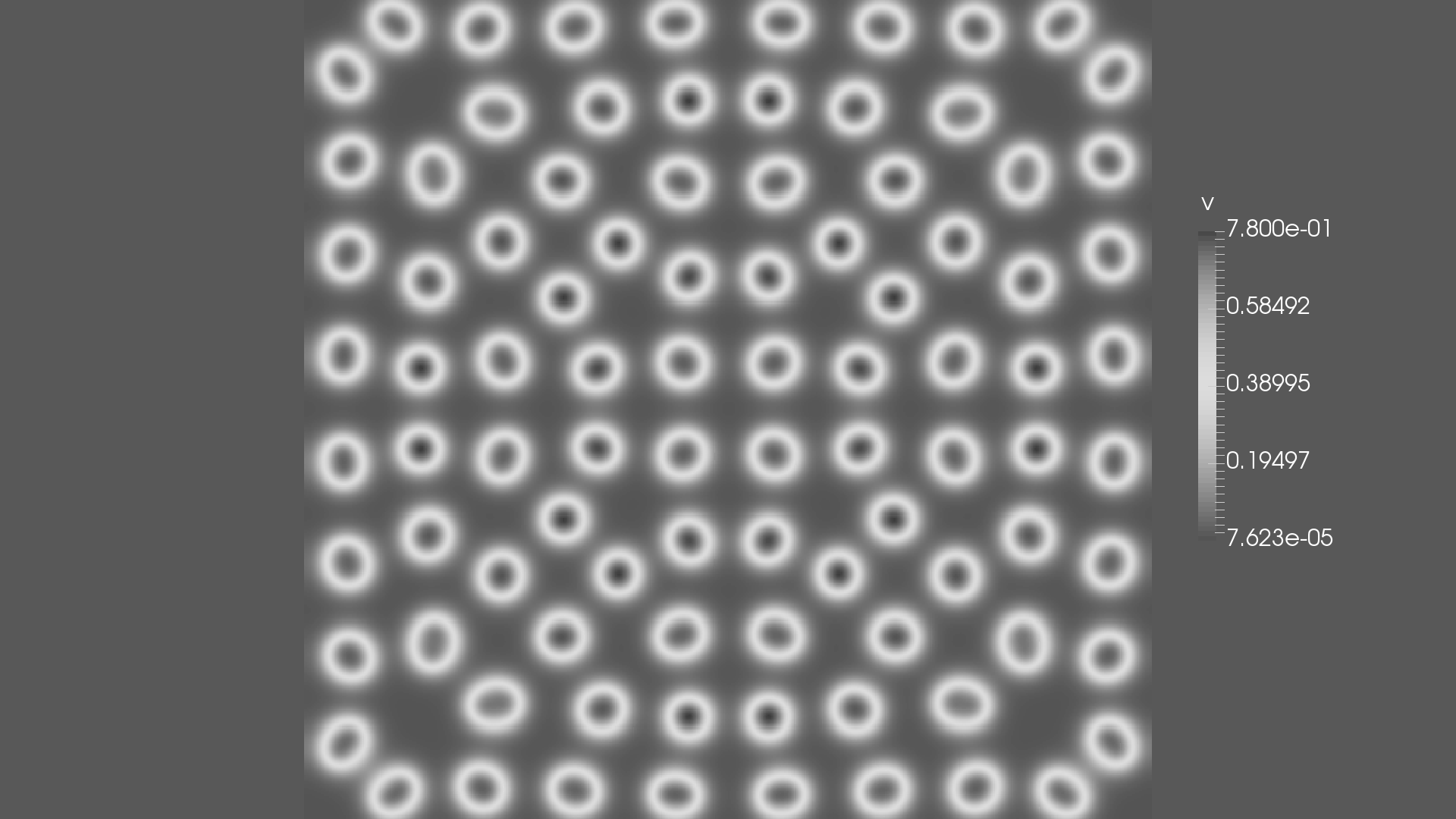}
\caption{Solution component $v$ at $t=0$ (left), $t=2000$ (middle) and $t=4000$ (right) for~\eqref{grayscott}.\label{fig:gs_sol}}
\end{center}
\end{figure}

%\begin{figure}[h!]
%\begin{center}
%\includegraphics[width=5.5cm]{gs_u_t3000} \hspace{5mm} \includegraphics[width=5.5cm]{gs_v_t3000}
%\caption{Solution components $u$ (left) and $v$ (right) for~\eqref{grayscott} at $T=3000$.\label{fig:gs_sol}}
%\end{center}
%\end{figure}

The problem can also naturally be stated in three spatial dimensions and solved by our methods. In Figure~\ref{fig:gs3d} we show the component $v$
computed by a complex embedded 4/3 splitting pair from~\cite{knth10b} with an underlying spatial discretization with $512^3$ basis functions
and a tolerance of $10^{-5}$.
The solution is plotted at times $t=250\begin{rev}0\end{rev}, \ t=300\begin{rev}0\end{rev},$
$t=400\begin{rev}0\end{rev},$ and $t=500\begin{rev}0\end{rev}$. In the following we will only investigate
the 2D case, as this does not influence the assessment of the time integrators, but reduces computation time.

\begin{figure}[h!]
\begin{center}
\includegraphics[width=8.8cm]{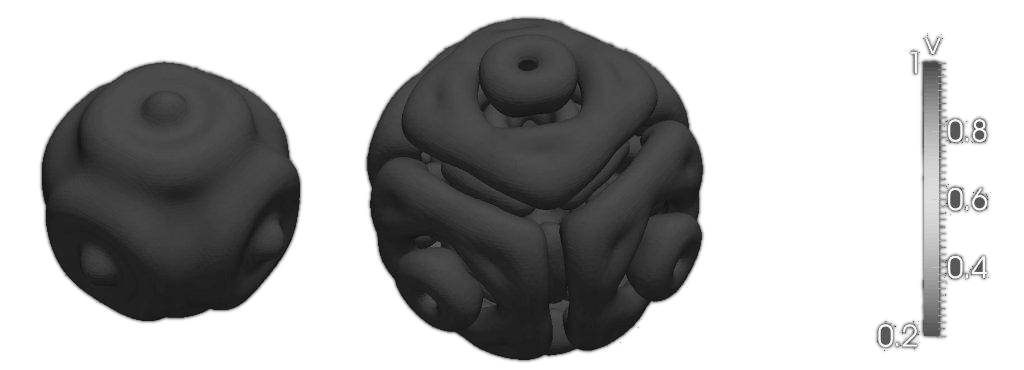} \\[4mm]
\includegraphics[width=4cm]{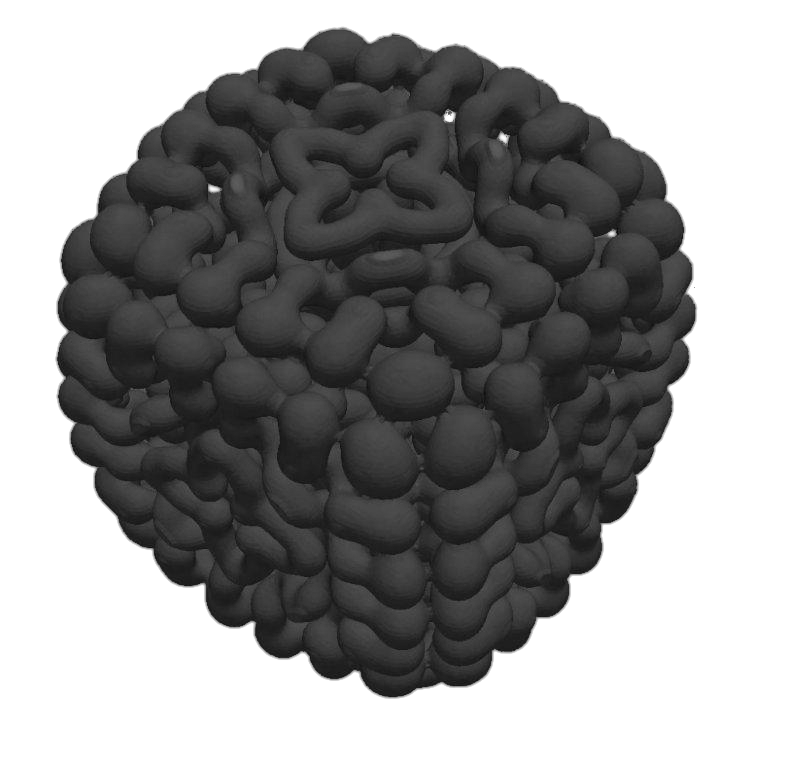} \hspace{1mm}
\includegraphics[width=4.15cm]{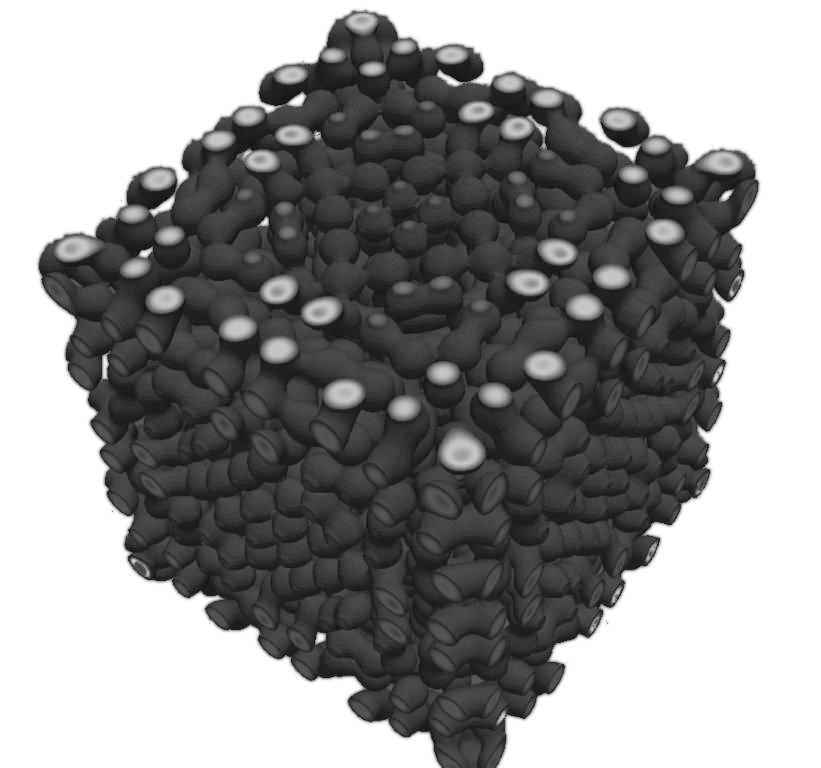}
\caption{Solution component $v$ for~\eqref{grayscott} in 3D at times $t=2500,\ 3000,\ 4000,\ 5000$.\label{fig:gs3d}}
\end{center}
\end{figure}

\subsection{Convergence analysis}\label{convgs}

For the theoretical analysis of the convergence of splitting methods, we use
the error representation~\eqref{errorexpansion}. Since the flow induced by the
cubic nonlinearity is not unconditionally stable, we have to resort to the
three-stage argument first given in~\cite{lubich07} for the cubic Schr{\"o}dinger
equation, see also~\cite{knth10a}:
\begin{itemize}
\item First, show stability in the $H^2$-norm.
\item The local error in $H^2$ is $O(h^{p-1})$, where the constant depends on
the $H^{2p-2}$-norm of $u$.
\item Stability together with consistency in $H^2$ implies convergence of order $p-2$ in $H^2$.
\item Convergence implies boundedness of the numerical solution in $H^2$.
\item Analyze stability in $H^1$. It turns out that the constant depends on the $H^2$-norms
of both the exact and the numerical solution. The latter has already been demonstrated
to be bounded.
\item The local error in $H^1$ is $O(h^{p})$, where the constant depends on
the $H^{2p-1}$-norm of $u$.
\item Since $\|u_n\|_{H^2}$ is bounded, stability and consistency imply convergence order $p-1$ in~$H^1$.
\item Analyze stability in $L^2$. It turns out that the constant depends on the
$H^2$-norms of both the exact and the numerical solution. The latter has already been demonstrated
to be bounded.
\item The local error in $L^2$ is $O(h^{p+1})$, where the constant depends on
the $H^{2p}$-norm of $u$.
\item We conclude convergence of order $p$ in $L^2$.
\end{itemize}
Along this line, we can prove the following theorem, since for the present situation
of a parabolic problem under periodic boundary conditions, the same Sobolev embeddings
hold as on the full space $\mathbb{R}^3$, see Appendix~\ref{sec:period}, so in particular the second order
differential operators and the cubic terms and their commutators admit the same bounds. Thus, the
following proof strategy can be followed in the same manner, taking into
account the commutator bounds given later:

\begin{theorem}\label{theorem1}
Suppose that the Gray--Scott equation~\eqref{grayscott} possesses a uniquely determined sufficiently regular solution~$u$ on the time interval~$[0,T]$.
Then, for any exponential operator splitting method~\eqref{splitting1} of (nonstiff) order~$p \geq 2$, the following error estimates are valid.
\begin{enumerate}
\item
Provided that $\norm{H^{2p}}{\, u(t)} \leq M_{2p}$ for $0 \leq t \leq T$, the bound
\begin{equation}
\label{NLSboundXzero}
\normbig{L^2}{\, u_n - u(t_n)} \leq C \, h^{p}\,, \qquad 0 \leq n \leq N\,, \quad t_N \leq T\,,
\end{equation}
holds true with constant~$C$ depending on~$M_{2p}$.
\item
Provided that $\norm{H^{2p-1}}{\, u(t)} \leq M_{2p-1}$ for $0 \leq t \leq T$, the bound
\begin{equation}
\label{NLSboundXone}
\normbig{H^1}{\, u_n - u(t_n)} \leq C \, h^{p-1}\,, \qquad 0 \leq n \leq N\,, \quad t_N \leq T\,,
\end{equation}
holds true with constant~$C$ depending on~$M_{2p-1}$.
\item
Provided that $\norm{H^{2p-2}}{\, u(t)} \leq M_{2p-2}$ for $0 \leq t \leq T$, the bound
\begin{equation}
\label{NLSboundXtwo}
\normbig{H^2}{\, u_n - u(t_n)} \leq C \, h^{p-2}\,, \qquad 0 \leq n \leq N\,, \quad t_N \leq T\,,
\end{equation}
holds true with constant~$C$ depending on~$M_{2p-2}$.
\end{enumerate}
\end{theorem}

\begin{proof}
We work out the analysis in detail for the case $p=2$,
the general case is proven analogously. For the analysis, we write the Gray--Scott system in the partitioned form
\begin{equation}\label{gs}
\pat U(x,y,t) = AU(x,y,t) + B(U(x,y,t)),\qquad U(x,y,0) = U_0(x,y),\qquad (x,y) \in[-\pi,\pi]^2,
\end{equation}
where
\begin{eqnarray*}
&& U(x,y,t) = \left(\begin{array}{c}
                u(x,y,t)\\
                v(x,y,t)\end{array}\right),\\
&&   A\,U(x,y,t) = \left( \begin{array}{cc}
                c_1 \Delta - \alpha & 0 \\
                0 & c_2 \Delta - \beta
             \end{array} \right) U(x,y,t)
             +
             \left( \begin{array}{c}
                \alpha \\
                0
             \end{array} \right), \\
&&   B(U(x,y,t)) = \left( \begin{array}{c}
                -u(x,y,t)v^2(x,y,t) \\
                u(x,y,t)v^2(x,y,t)
             \end{array} \right).
\end{eqnarray*}

Stability is shown in the same manner as for the cubic Schr{\"o}dinger equation~\cite{knth10a},
see the outline above.
To bound the local error, we compute the commutators of the vector fields. This yields
\begin{eqnarray*}
[A,B](U) &=& AB(U) - B'(U)AU \\
         &=& \left( \begin{array}{cc}
                    c_1 \Delta - \alpha & 0 \\
                    0 & c_2 \Delta - \beta
                    \end{array} \right)
             \left( \begin{array}{c}
                    -u v^2 \\
                    u v^2
                    \end{array} \right) + \\
         && + \left(\begin{array}{cc}
                    v^2 & 2uv \\
                    -v^2 & -2uv
                    \end{array} \right)
              \left( \begin{array}{c}
                    (c_1 \Delta - \alpha)u \\
                    (c_2 \Delta - \beta) v
                     \end{array} \right)  \\
         &=& \left( \begin{array}{c}
                    -c_1 \Delta(uv^2) + v^2 c_1 \Delta u + 2 uv(c_2 \Delta v - \beta v) \\
                    (c_2 \Delta - \beta) uv^2 - v^2(c_1 \Delta - \alpha) u - 2uv(c_2\Delta - \beta) v
                    \end{array} \right)
\\
         &=& \left( \begin{array}{c}
                    2(c_2-c_1)uv\Delta v -4c_1 v \nabla u \cdot \nabla v -2c_1 u \nabla v \cdot \nabla v +2c_1 \beta u v^2 \\
                    (c_2 - c_1) v^2\Delta u + 4 c_2 v \nabla u\cdot \nabla v +2c_2 u \nabla v \cdot \nabla v + (\alpha+\beta) u v^2
                    \end{array} \right).
\end{eqnarray*}
This can be estimated in Sobolev norms by resorting to the embeddings in Appendix~\ref{sec:sobemb}:
\begin{eqnarray}
\left\| [A,B](U) \right\|_{H^m} &\leq& C(\| U \|_{H^{m+2}} ),\quad m=0,1,\dots. \label{komm1hm}
\end{eqnarray}
For the second commutator we compute
\begin{eqnarray*}
B'(U)W &=& \left( \begin{array}{cc}
                    -v^2 & -2uv \\
                     v^2 &  2uv
                  \end{array} \right)
            \left( \begin{array}{c} w_1 \\ w_2 \end{array} \right) \\
       &=& \left( \begin{array}{c}
                    -v^2 w_1 - 2uvw_2 \\
                    v^2 w_1 + 2 u v w_2
                  \end{array} \right) \\
B''(U)(W,Z) &=& \left( \begin{array}{cc}
                    0 & -2uw_2 \\
                    0 &  2uw_1
                  \end{array} \right)
            \left( \begin{array}{c} z_1 \\ z_2 \end{array} \right) \\
            &=& \left( \begin{array}{c}
                    -2uw_2z_1 \\
                    2uw_1z_2
                  \end{array} \right), \\
A^2\,U &=& \left( \begin{array}{c}
                     (c_1 \Delta - \alpha)^2 u \\ (c_2 \Delta - \beta)^2 v
                 \end{array} \right) +
                 \left( \begin{array}{c}
                     \alpha (c_1 \Delta - \alpha +1 ) \\ 0
                 \end{array} \right)
\end{eqnarray*}
and hence
\begin{eqnarray*}
[A,[A,B]](U) &=& A^2 B(U) - 2AB'(U)AU + B''(U)(AU,AU) + B'(U)A^2U
\end{eqnarray*}
contains terms of the form $uv\Delta^2 u$ and $uv\Delta^2v$ which do not cancel.
Consequently,
\begin{eqnarray}
\left\| [A,[A,B]](U) \right\|_{H^m} &\leq& C(\| U \|_{H^{m+4}} ),\quad m=0,1,\dots.  \label{komm2hm}
\end{eqnarray}
Inductively, the result for higher commutators appearing in estimates for
higher-order splitting methods follows. \qquad \qed
\end{proof}

\subsection{Numerical results}

In this section, we will demonstrate the accuracy of several splitting schemes for the Gray--Scott equation
\eqref{grayscott} by computing the convergence orders
with an underlying Fourier pseudospectral space discretization at $512\times512$ points. {{The nonlinear terms in the equation are propagated
using an explicit fourth order Runge-Kutta method.}} For these experiments, the parameters in
\eqref{grayscott} were chosen as $\alpha=0.038,\ \beta=0.114,\ c_1= 0.04,\ c_2=0.005.$
%and the initial condition was $U(x,0) = \left(0.5 + \mathrm{e}^{-1-(x_1^2+x_2^2+x_3^2)}, 0.1 + \mathrm{e}^{-1-(x_1^2+x_2^2+x_3^2)}\right)^T.$
We will investigate the pair \cite[\scheme{Milne\;2/2\;c\;(i)}]{splithp}, and the optimized palindromic fourth order method
\cite[\scheme{Emb\;4/3\;A\;c}]{splithp}. The error estimators are based on the Milne device (Section~\ref{subsec:milne}), and
the embedding idea (Section~\ref{subsec:embed}), respectively. Figure~\ref{fig:milne22ci} gives
the error of the method \cite[\scheme{Milne\;2/2\;c\;(i)}]{splithp} and the error of the associated
error estimator as well as the global error of the time integration. The empirical convergence order
can be observed by comparing the computed data points with the solid line representing the theoretical order
extrapolated from the most accurate approximation. Figure~\ref{fig:emb43} gives the same data
for the integrator from \cite[\scheme{Emb\;4/3\;A\;c}]{splithp} and associated error estimator. Errors are calculated with respect to a reference solution computed by \cite[\scheme{Emb\;4/3\;A\;c}]{splithp} with time-step $h=7.81\cdot 10^{-3}$.
The empirical orders illustrate the theoretical result in Theorem~\ref{theorem1}.

\begin{figure}[h!]
\begin{center}
\includegraphics[width=8.5cm]{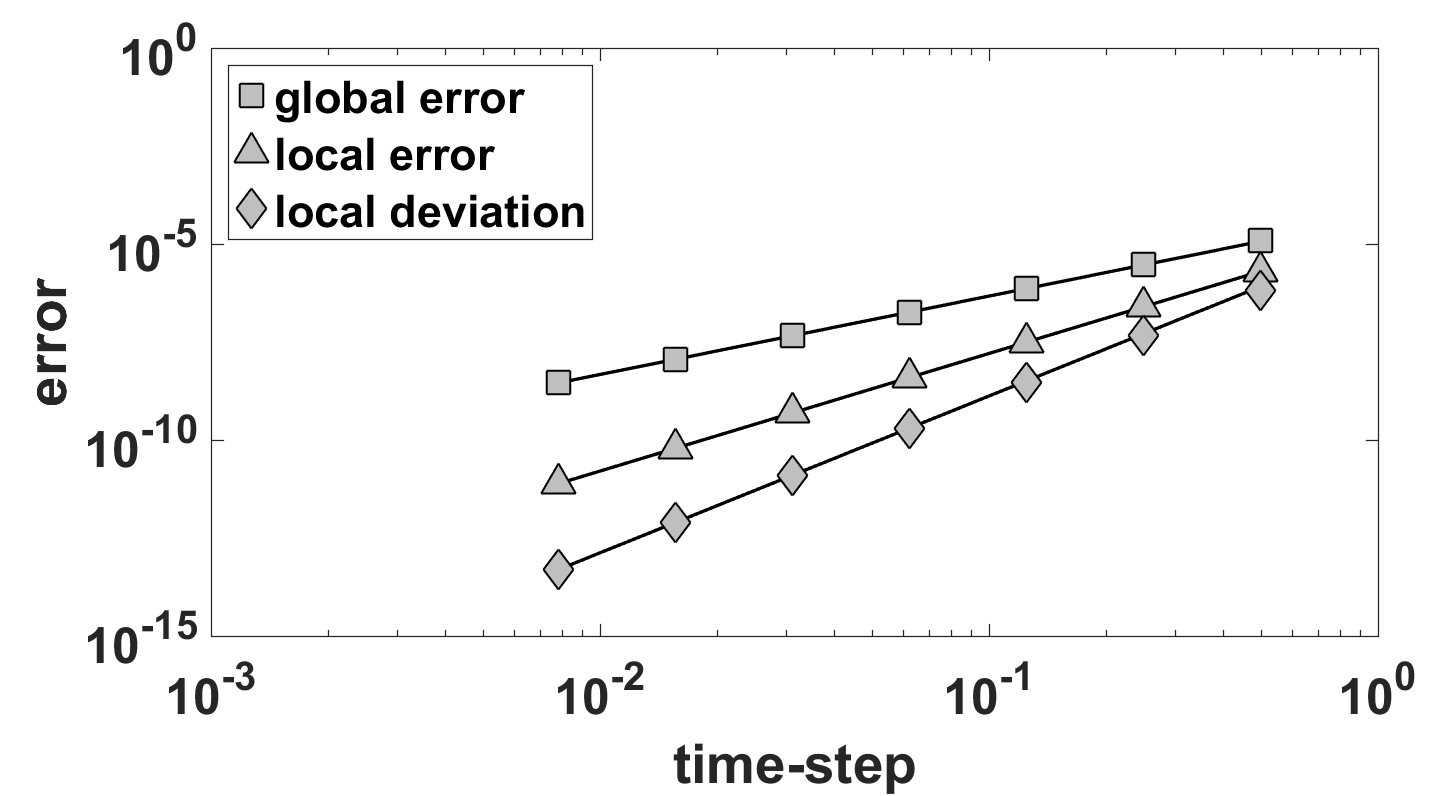}
\caption{Empirical convergence orders of the local and the global errors
and deviation of the local error estimator
for the \cite[\scheme{Milne\;2/2\;c\;(i)}]{splithp} splitting applied to the
Gray--Scott equation~\eqref{grayscott}.\label{fig:milne22ci}}
\end{center}
\end{figure}

\begin{figure}[h!]
\begin{center}
\includegraphics[width=8.4cm]{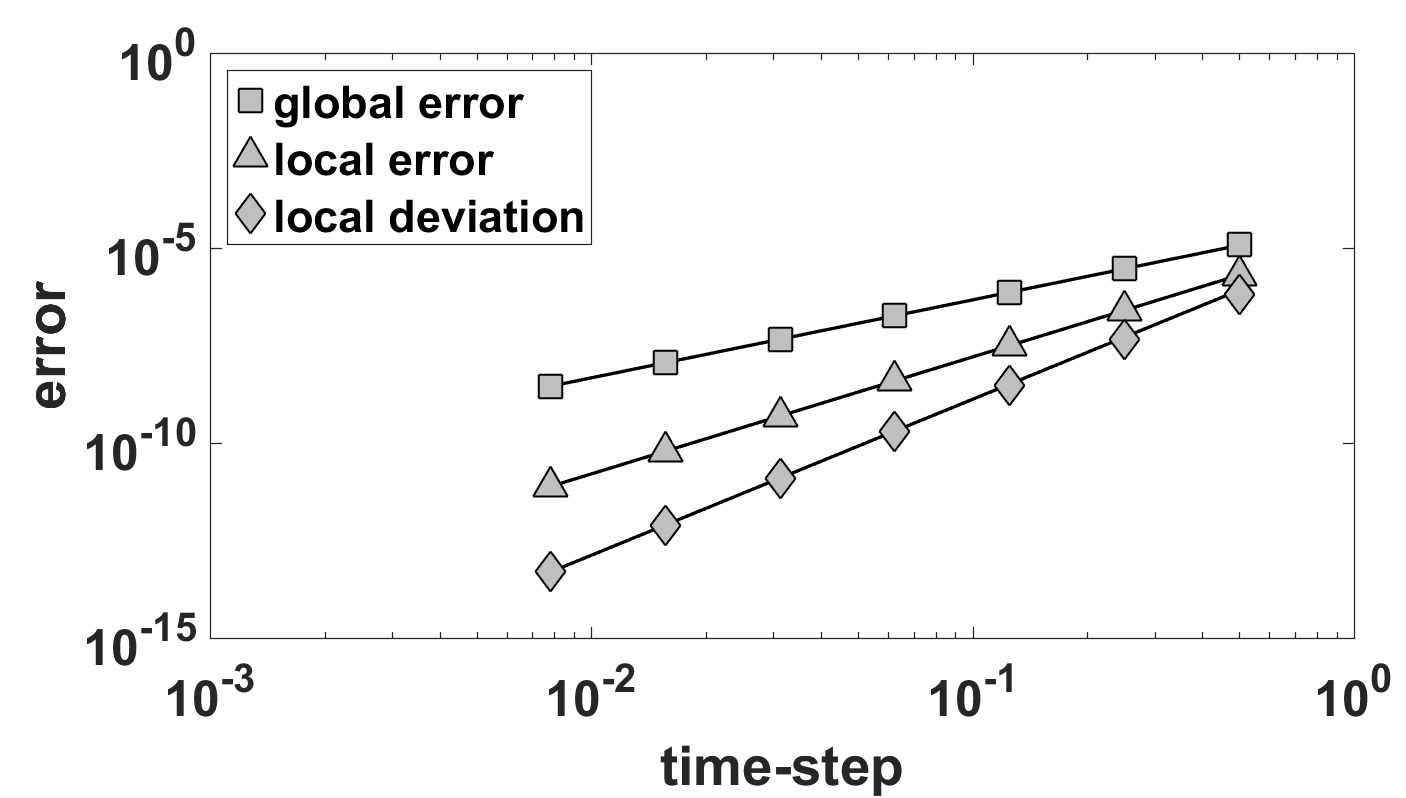}
\caption{Empirical convergence orders of the local and the global errors for the \cite[\scheme{Emb\;4/3\;A\;c}]{splithp} splitting
and deviation of the local error estimator applied to the
Gray--Scott equation~\eqref{grayscott}.\label{fig:emb43}}
\end{center}
\end{figure}

The time-steps generated in the course of an adaptive procedure are given in Figure~\ref{table:steps-gs}.
The left plot shows the time-steps to satisfy a tolerance of $10^{-5}$ for the \cite[\scheme{Milne\;2/2\;c\;(i)}]{splithp}
method, and likewise on the right for the \cite[\scheme{Emb\;4/3\;A\;c}]{splithp} pair.

\begin{figure}[h!]
\begin{center}
\includegraphics[width=8.1cm]{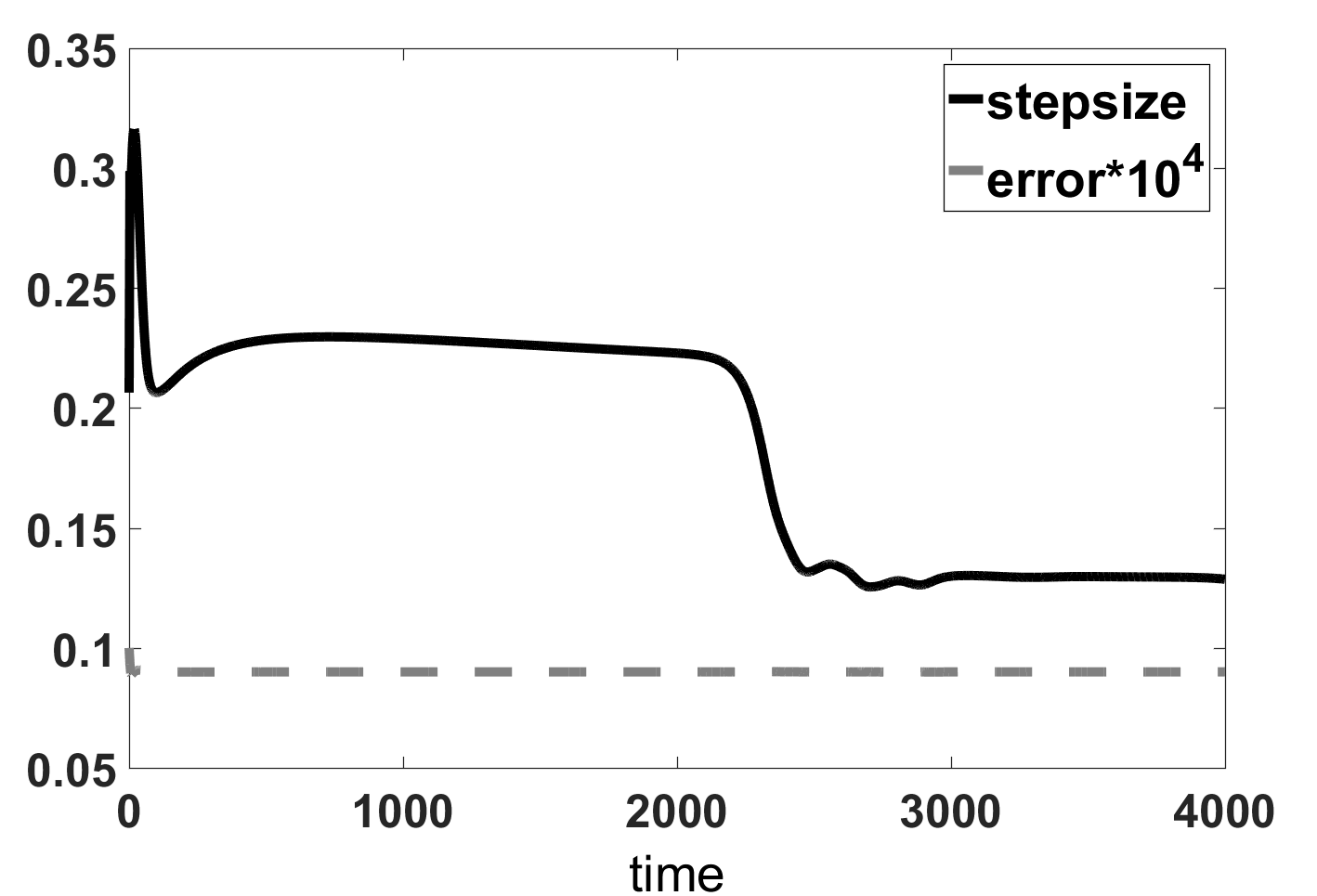}

\vspace{3mm}

\includegraphics[width=8.1cm]{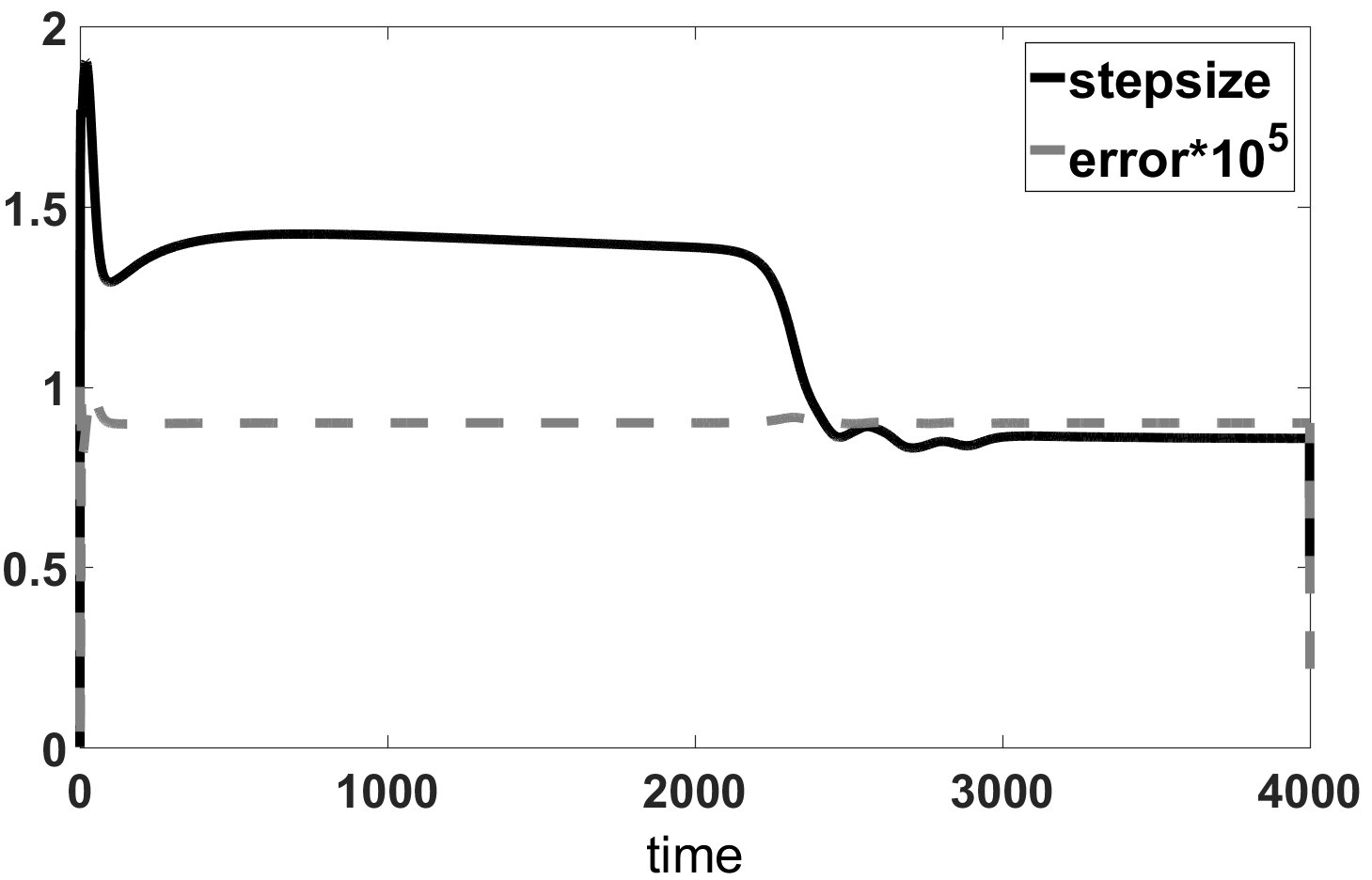}
\caption{Time-steps for~\eqref{grayscott} for \cite[\scheme{Milne\;2/2\;c\;(i)}]{splithp} (top),
\cite[\scheme{Emb\;4/3\;A\;c}]{splithp} (bottom), tolerance $10^{-5}$.\label{table:steps-gs}}
\end{center}
\end{figure}

\subsection{Comparisons}\label{subsec:compare_gs}

After verifying the reliability of the investigated solution methods, we will assess
the efficiency of the adaptive time integration methods by giving a comparison to
the situation where the same accuracy is achieved with constant time-steps.
Moreover, we will compare the efficiency of adaptive time integration
based on the second order method in conjunction with the Milne device
as compared to the fourth order embedded splitting pair \cite[\scheme{Emb\;4(3)\;A\;c}]{splithp}
and the palindromic scheme \cite[\scheme{PP\;3/4\;A\;c}]{splithp}. By construction, the latter also
provides an asymptotically correct error estimator, which by its special
structure is cheap to evaluate.
\begin{rev}Runtime was measured on a PC
with Intel Core i7-2600 3, 4GHz Quad-Core processor with 16 GB RAM:\end{rev}
Table~\ref{efficiencygs} shows the number of steps required in the adaptive
integration, the number of equidistant steps with the smallest necessary
adaptive time-step, and the computing time for both scenarios. The tolerances
were chosen as $10^{-5}$ (top) and $10^{-8}$ (bottom), respectively.
We observe that indeed the adaptive methods require fewer steps, but the
overall computational cost is higher due to the effort for the evaluation of the
error estimator in each step. This suggests an adaptive strategy which
does not estimate the error in each step, but only after a certain number
of steps with a fixed time-step. This is also supported by the
fact that a measurement of the computation time for the \cite[\scheme{Milne\;2/2\;c\;(i)}]{splithp}
method on 1000 equidistant steps yielded $75.45$ seconds, in conjunction
with the error estimator the computation time amounted to $123.43$ seconds.
The same experiment for the \cite[\scheme{Emb\;4/3\;A\;c}]{splithp} method yielded
$162.18$ and $238.36$ seconds, respectively.
For \cite[\scheme{PP\;3/4\;A\;c}]{splithp} the runtimes were $110.75$ seconds and $193.68$ seconds, respectively.
This implies that an update
of the time-steps every two or three steps should provide a more efficient
strategy, but possibly at the cost of reduced numerical stability,
\begin{rev}
since this example shows rather smooth solution dynamics. Indeed,
the step-size is adjusted rapidly by exploiting the maximally permitted
increase by a factor of 4 from a too small initial guess
to the appropriate value, which is assumed throughout the rest
of the computation, see Figure~\ref{fig:change_gs}, which gives the
quotient of two consecutive step-sizes over the integration interval.
This behavior demonstrates one major advantage of adaptivity, that
an unsuitable initial guess of the step-size is automatically adjusted to an optimal value.
\end{rev}

\begin{table}[h!]
\begin{center}
\begin{tabular}{||r||r|r|r|r||}
\hline
Method & \# steps adaptive & \# steps equidist & time adaptive & time equidist \\
\hline
\scheme{Milne\;2/2\;c\;(i)}, \textrm{tol}$=10^{-5}$ & $406$ & $486$ & $57.04$ & $28.21$ \\
\scheme{Emb\;4/3\;A\;c}, \textrm{tol}$=10^{-5}$     & $67$ & $79$ & $17.72$ & $11.46$ \\
\scheme{PP\;3/4\;A\;c}, \textrm{tol}$=10^{-5}$      & $116$ & $135$ & $23.02$ & $12.99$ \\ \hline
\scheme{Milne\;2/2\;c\;(i)}, \textrm{tol}$=10^{-8}$ & $4691$ & $5625$ & $878.72$ & $503.93$ \\
\scheme{Emb\;4/3\;A\;c}, \textrm{tol}$=10^{-8}$     & $516$ & $612$ & $174.30$ & $128.19$ \\
\scheme{PP\;3/4\;A\;c}, \textrm{tol}$=10^{-8}$      & $929$ & $1107$ & $195.87$ & $106.79$ \\
\hline
\end{tabular}
\caption{Comparison of the efficiency of \cite[\scheme{Milne\;2/2\;c\;(i)}]{splithp},
\cite[\scheme{Emb\;4/3\;A\;c}]{splithp}, and \cite[\scheme{PP\;3/4\;A\;c}]{splithp}
for~\eqref{grayscott}.
The tolerances were $10^{-5}$ (top) and $10^{-8}$ (bottom), respectively.\label{efficiencygs}}
\end{center}
\end{table}

\begin{figure}
\begin{center}
\includegraphics[width=8cm]{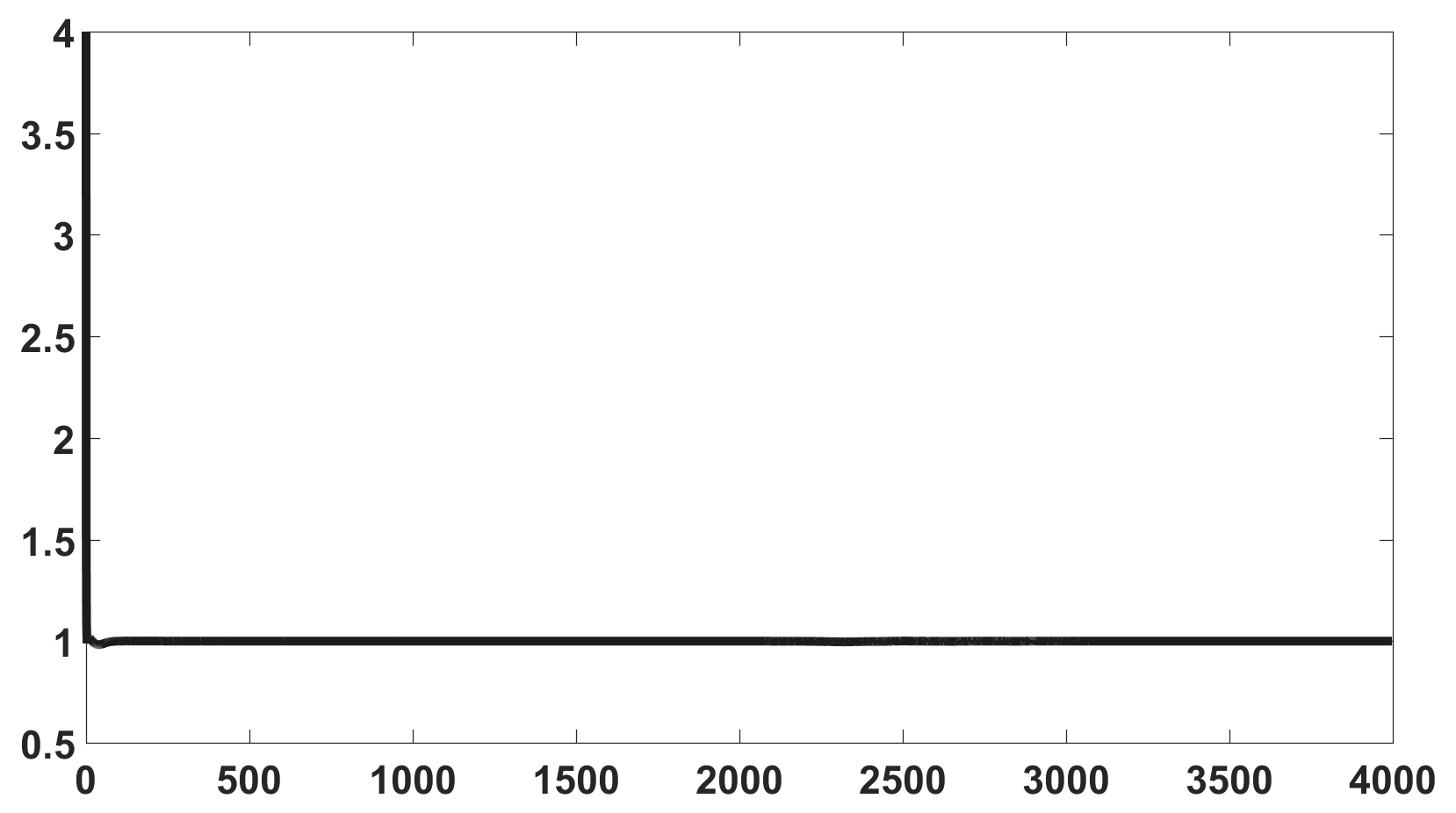}
\caption{Ratio of two consecutive time step-sizes for the solution of~\eqref{grayscott} by
\cite[\scheme{Emb 4/3 A c}]{splithp}.\label{fig:change_gs}}
\end{center}
\end{figure}

\section{The Van der Pol equation}\label{sec:VP}

The Van der Pol equation is an ordinary differential equation with limit cycle behavior. It is used as a test of time integration schemes for stiff differential equations. It shares characteristics with simple models for cardiac behavior. The Van der Pol equation is usually considered as an ordinary differential equation, but by adding diffusion terms, one can consider an extension from a set of ordinary differential equations to a pair of coupled partial differential equations with spatial dependence.

It is given by
\begin{subequations}\label{VanderPol}
\begin{eqnarray}
\hspace{-9mm}&& \pat u(x,t) = D_u\, \Delta u(x,t) +  \,  v(x,t), \\
\hspace{-9mm}&& \pat v(x,t) = D_v\, \Delta v(x,t) + \,  \frac{1}{\epsilon}\left[ ( 1-u^2(x,t)) \, v(x,t)  - u(x,t) \right].
\end{eqnarray}
\end{subequations}
It is split into
\begin{subequations}\label{VanderPolA}
\begin{eqnarray}
\hspace{-9mm}&& \pat u(x,t) = D_u\, \Delta u(x,t) +  \,  v(x,t), \\
\hspace{-9mm}&& \pat v(x,t) =  D_v\, \Delta v(x,t) + \, \frac{1}{\epsilon} ( v(x,t) - u(x,t) ) ,
\end{eqnarray}
\end{subequations}
and
\begin{subequations}\label{VanderPolB}
\begin{eqnarray}
\hspace{-9mm}&& \pat v(x,t) =  - \frac{1}{\epsilon}u^2(x,t) \, v(x,t).
\end{eqnarray}
\end{subequations}
\begin{rev}
The convergence result for an order $p$ splitting applied to this system can
readily be seen to be the same as Theorem~\ref{theorem1}. However, the constants
in the estimates \eqref{NLSboundXzero}--\eqref{NLSboundXtwo} depend on
the small parameter $\epsilon$, $C=C(M_{2p},\epsilon^{-p})$ in
\eqref{NLSboundXzero}, $C=C(M_{2p-1},\epsilon^{1-p})$
in~\eqref{NLSboundXone},
and $C=C(M_{2p-2},\epsilon^{2-p})$ in~\eqref{NLSboundXtwo}.
We must stress that the involved estimates of the exact solution
will also be negatively influenced when $\epsilon$ is small.
The analysis of the exact solution is not a topic of the present paper, however.

For our comparisons, we solve the problem in one spatial dimension, with $x\in [-\pi,\pi]$, and choose $\epsilon=10^{-3}$.
The evolution of the solution components with $t$ (on the vertical axis) is illustrated in Figure~\ref{figure:vp-pde}.
Results showing the effectiveness of adaptive time stepping for~\eqref{VanderPol} are shown in Table~\ref{efficiencyVPpde}.
For this problem, the lower order method is more efficient. Adaptive step selection
yields a speed-up by about a factor 5. Indeed, if we consider the ratio of two consecutive
step-sizes, we see some variation in the region of the steep layers in Figure~\ref{fig:change_vdp},
which is obviously sufficiently large to warrant adaptive time-stepping.

\begin{table}[h!]
\begin{center}
\begin{tabular}{||r||r|r|r|r||}
\hline
Method & \# steps adaptive & \# steps equidist & time adaptive & time equidist \\
\hline
\scheme{PP\;5/6\;A\;c}, \textrm{tol}$=10^{-3}$ & $ 11886 $ & $ 127118 $ & $ 6.41\mathrm{e}+03 $ & $ 3.14\mathrm{e}+04 $ \\
\scheme{PP\;3/4\;A\;c}, \textrm{tol}$=10^{-3}$ & $ 20989 $ & $ 217760 $ & $ 4.32\mathrm{e}+03 $ & $ 2.09\mathrm{e}+04 $ \\
\hline
\scheme{PP\;5/6\;A\;c}, \textrm{tol}$=10^{-5}$ & $ 124559 $ & $ 1269018 $ & $ 6.20\mathrm{e}+04 $ & $ 3.13\mathrm{e}+05 $ \\
\scheme{PP\;3/4\;A\;c}, \textrm{tol}$=10^{-5}$ & $ 214338 $ & $ 2176945 $ & $ 3.81\mathrm{e}+04 $ & $ 2.09\mathrm{e}+05 $\\
\hline
\end{tabular}
\caption{Comparison of the efficiency of \cite[\scheme{PP\;5/6\;A\;c}]{splithp}, and
\cite[\scheme{PP\;3/4\;A\;c}]{splithp}
for~\eqref{VanderPol} with $D_u=D_v=1$ and $\epsilon=10^{-3}$, $256$ grid points . The final time was 10.0, with
initial condition $u(x,0)=\exp(-x^2)$ and $v(x,0)=0.2\exp(-(x+2)^2)$ and $x\in[-\pi,\pi]$.\label{efficiencyVPpde}}
\end{center}
\end{table}

\begin{figure}[h!]
\begin{center}
\hspace{-5mm}
\includegraphics[width=6.5cm]{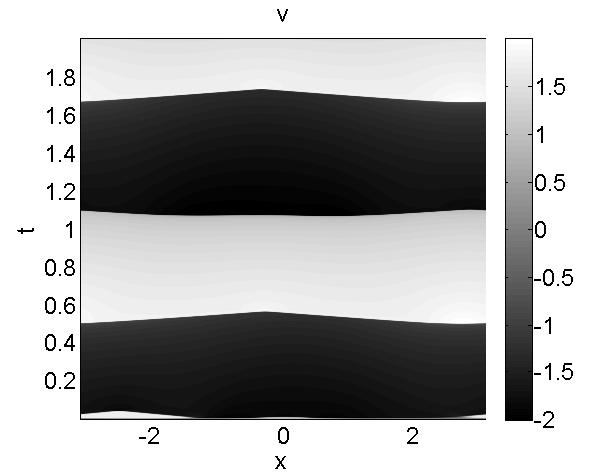}
\hspace{-5mm}
\includegraphics[width=6.5cm]{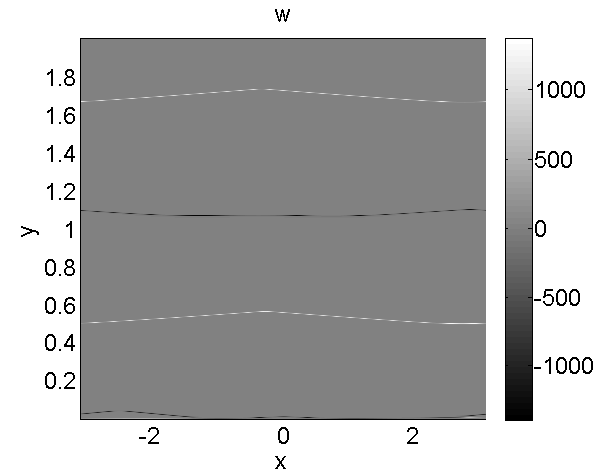}
\caption{Numerical solution for~\eqref{VanderPol} with $D_u=D_v=1$ and $\epsilon=10^{-3}$, $256$ grid points, with
initial condition $u(x,0)=\exp(-x^2)$ and $v(x,0)=0.2\exp(-(x+2)^2)$ and $x\in[-\pi,\pi]$ (left: $u$; right: $v$).\label{figure:vp-pde}}
\end{center}
\end{figure}

The time-steps generated in the course of an adaptive procedure are given in Figure~\ref{table:steps-vdp}.
The left plot shows the time-steps to satisfy a tolerance of $10^{-5}$ for the \scheme{PP 3/4 A c}
method, and likewise on the right for the \cite[\scheme{PP 5/6 A c}]{splithp} pair.

\begin{figure}
\begin{center}
\includegraphics[width=8cm]{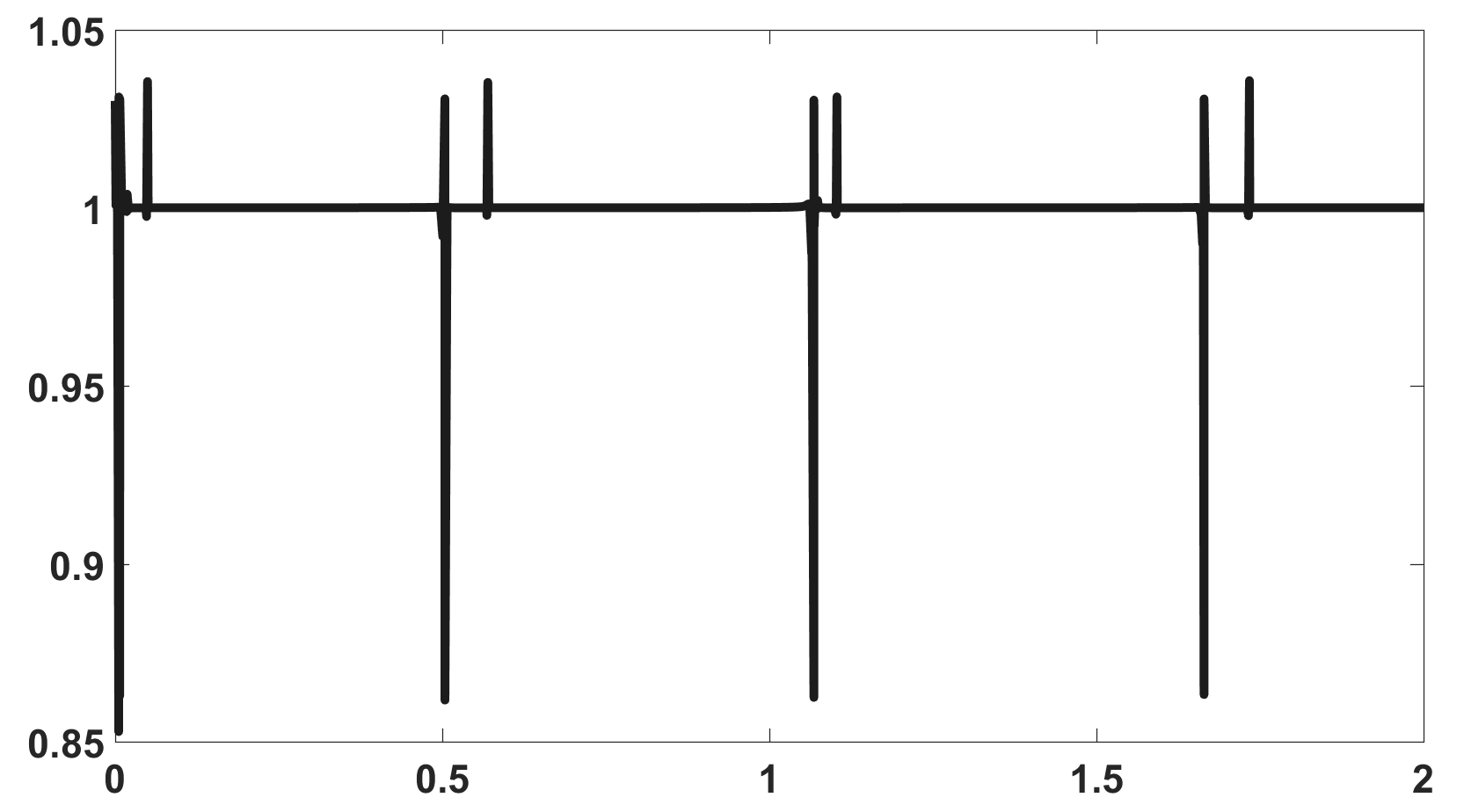}
\caption{Ratio of two consecutive time step-sizes for the solution
of~\eqref{VanderPol} by
\cite[\scheme{PP 5/6 A c}]{splithp}.\label{fig:change_vdp}}
\end{center}
\end{figure}

\begin{figure}[h!]
\begin{center}
\includegraphics[width=8.0cm]{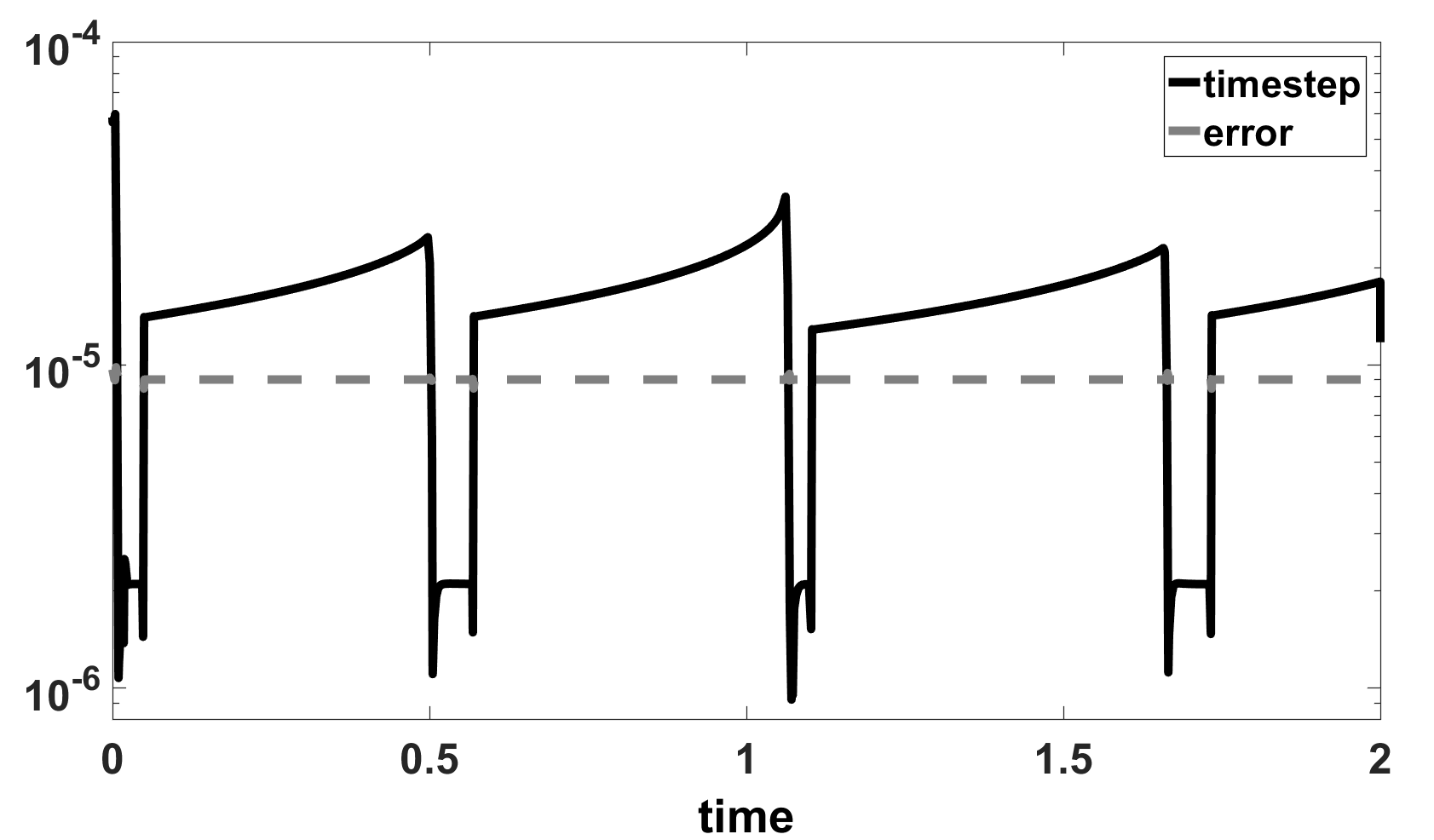}

\vspace{3mm}

\includegraphics[width=8.0cm]{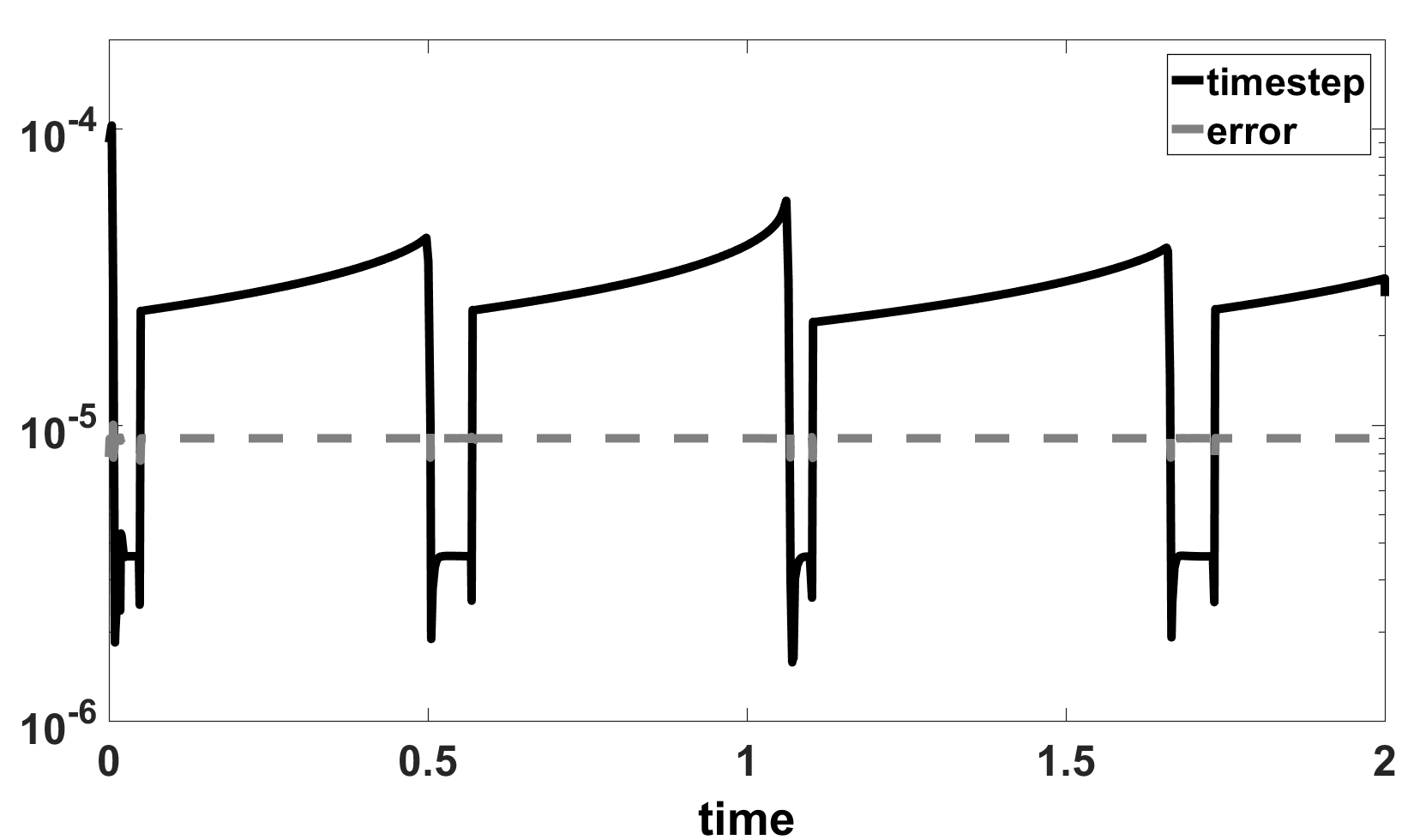}
\caption{Time-steps for~\eqref{VanderPol} for \cite[\scheme{PP 3/4 A c}]{splithp} (top),
\cite[\scheme{PP 5/6 A c}]{splithp} (bottom), tolerance $10^{-5}$.\label{table:steps-vdp}}
\end{center}
\end{figure}

\end{rev}

\section{Splitting into three operators (`$ABC$-splitting')}\label{sec:abc}
Finally, we consider a splitting of the Gray--Scott equations~\eqref{grayscott} into three parts,
\begin{equation*}
\underbrace{\left( \begin{array}{cc}
                c_1 \Delta - \alpha & 0 \\
                0 & c_2 \Delta - \beta
             \end{array} \right) U(x,y,t)
             +
             \left( \begin{array}{c}
                \alpha \\
                0
             \end{array} \right)}_{=A}
  +
  \underbrace{\left( \begin{array}{c} 0\\ u(x,y,t) v^2(x,y,t)\end{array}\right)}_{=B}
  -
  \underbrace{\left( \begin{array}{c} u(x,y,t) v^2(x,y,t)\\ 0\end{array}\right)}_{=C}.
\end{equation*}
This has the computational advantage that the flows of the operators $B$ and $C$ can be
computed analytically when the other component is frozen. Below we verify the convergence
orders for this case for the optimal palindromic splitting \scheme{PP\;3/4\;A\;3\;c}.

\begin{rev}
\textit{Remark:}\/ A formal error analysis for $ABC$-splitting has not yet been given
in the nonlinear case, the linear case has been treated in~\cite{auzingeretal14a}.
However, inspection of the commutators that would critically influence
the error shows that a convergence result analogous to Theorem~\ref{theorem1}
will hold, since commutators of $B$ and $C$ vanish.
\end{rev}

\subsection{Numerical results}

The numerical results below were computed by the method \cite[\scheme{PP\;3/4\;A\;3\;c}]{splithp}. This is the method
of order~3 with the smallest leading error coefficients (see~\cite{part1}) we could determine and offers
the advantage of the cheap error estimator from Section~\ref{subsec:palindromic},
\begin{rev}
see Figure~\ref{fig:ppabc}.
\end{rev}

\begin{figure}[h!]
\begin{center}
\includegraphics[width=12.0cm]{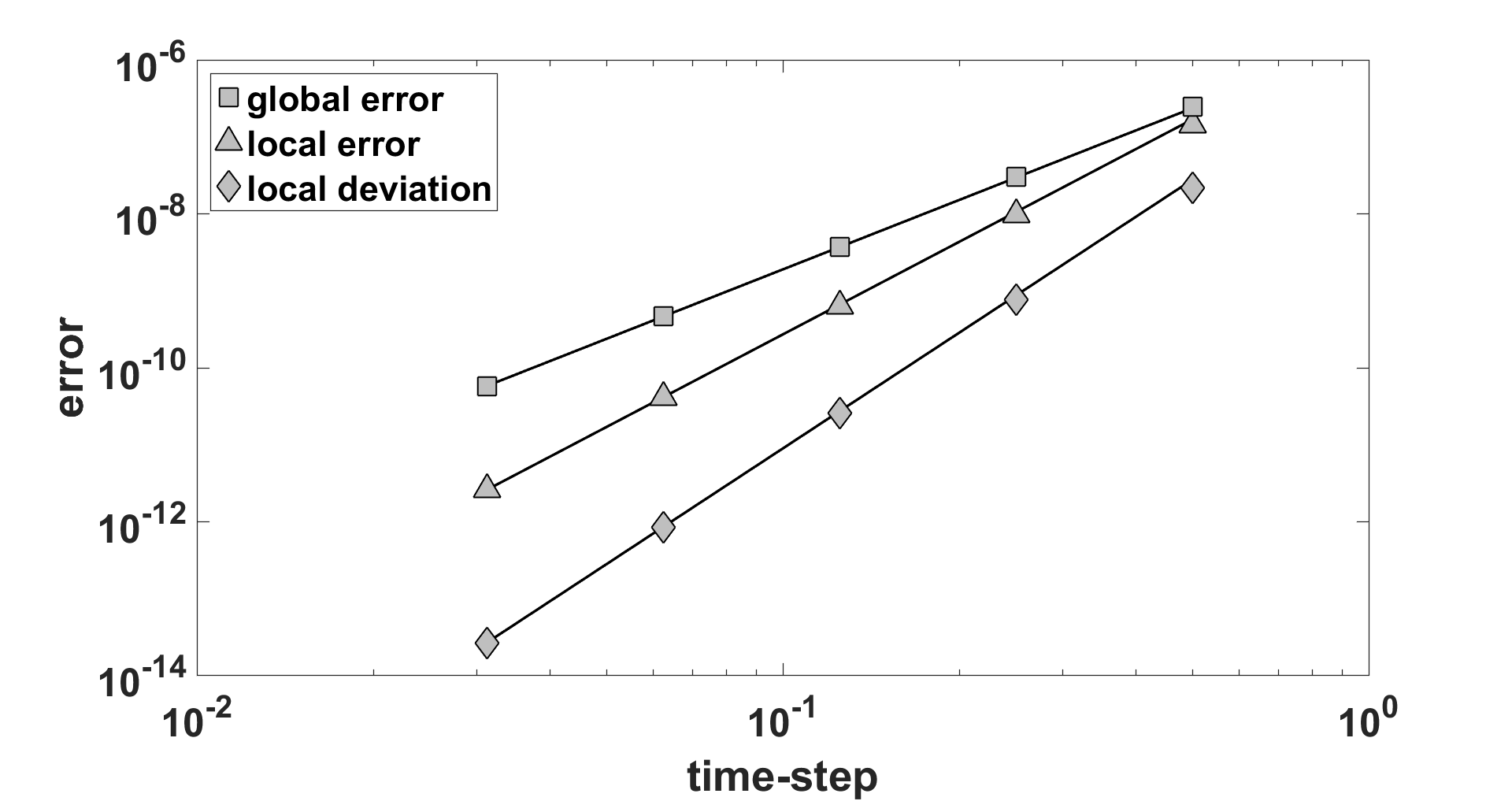}
\caption{Empirical convergence orders of the local and the global errors for
the \cite[\scheme{PP\;3/4\;A\;3\;c}]{splithp} splitting applied to the
Gray--Scott equation~\eqref{grayscott}.\label{fig:ppabc}}
\end{center}
\end{figure}

The time-steps generated in the course of an adaptive procedure
according to Section~\ref{subsec:adaptstep} are given in Figure~\ref{table:steps-gsabc}.
The plot shows the step-sizes to satisfy a tolerance of $10^{-5}$ for the \scheme{PP\;3/4\;A\;3\;c}
method.

\begin{figure}[h!]
\begin{center}
\includegraphics[width=8.1cm]{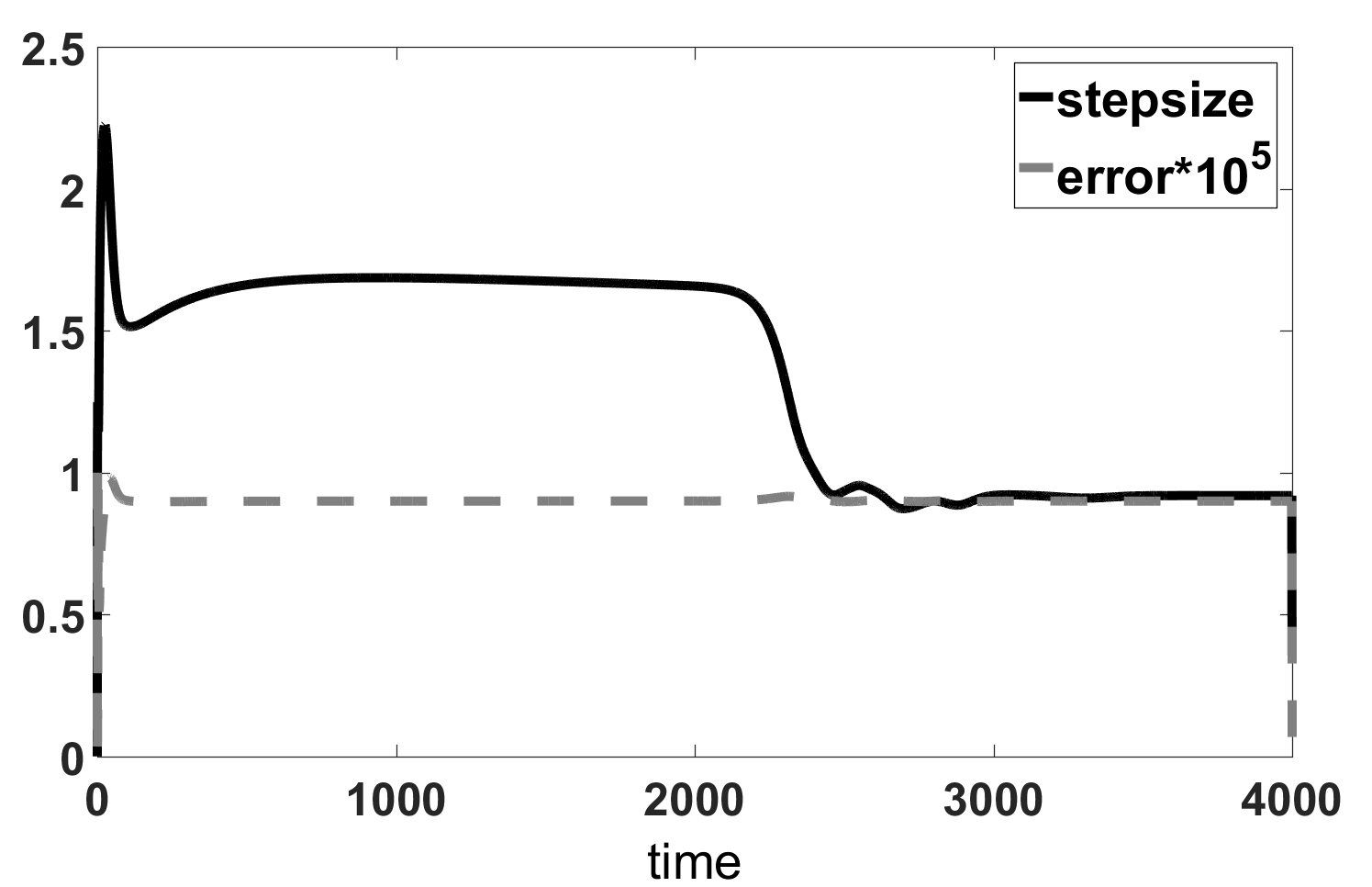}
\caption{Time-steps and local error for~\eqref{grayscott} for \cite[\scheme{PP\;3/4\;A\;3\;c}]{splithp}.\label{table:steps-gsabc}}
\end{center}
\end{figure}

\subsection{Comparisons}\label{subsec:compare_gsabc}

In order to compare the efficiency of the $ABC$-splitting approach with the two-operator splitting
discussed in Section~\ref{sec:gs}, in Table~\ref{efficiencygsabc} we give the number of steps
required for tolerances $10^{-5}$ and $10^{-8}$ and the resulting computation times.
It is observed that the $ABC$-splitting
\cite[\scheme{PP\;3/4\;A\;3\;c}]{splithp} requires slightly fewer steps than
\cite[\scheme{PP\;3/4\;A\;c}]{splithp}, but the
computation time is higher. The reason is that each individual step is computationally
more demanding in the $ABC$-splitting due to the larger number of required FFT transforms
associated with the larger number of compositions. These result from
the fact that the number of order conditions is larger in the $ABC$ case and
therefore, more free parameters are necessary to construct high-order methods.
Indeed, 1000 steps with \cite[\scheme{PP\;3/4\;A\;c}]{splithp} required $110.75$ seconds, for
\cite[\scheme{PP\;3/4\;A\;3\;c}]{splithp} the timing was $199.95$.
\begin{rev}
However, we stress again that a major advantage of the $ABC$-splitting approach for this
example lies in the fact that the computations of the nonlinear flows
can resort to analytical solutions instead of numerical approximations as in Section~\ref{sec:gs}.
\end{rev}

{
\begin{table}[h!]
\begin{center}
\begin{tabular}{||r||r|r|r|r||}
\hline
Method & \# steps adaptive & \# steps equidist & time adaptive & time equidist \\
\hline
{ \scheme{PP\;3/4\;A\;3\;c}, \textrm{tol}$=10^{-5}$ }    & $65$ & $67$ & $26.74$ & $13.23$ \\ \hline
{ \scheme{PP\;3/4\;A\;3\;c}, \textrm{tol}$=10^{-8}$ }    & $555$ & $645$ & $244.41$ & $138.38$ \\
\hline
\end{tabular}
\caption{Efficiency of the \cite[\scheme{PP\;3/4\;A\;3\;c}]{splithp} splitting
for~\eqref{grayscott}. The tolerances were $10^{-5}$ (top) and $10^{-8}$ (bottom), respectively.\label{efficiencygsabc}}
\end{center}
\end{table}
}

%{
%\begin{table}[h!]
%\begin{center}
%\begin{tabular}{||r||r|r|r|r||}
%\hline
%Method & \# steps adaptive & \# steps equidist & time adaptive & time equidist \\
%\hline
%{ \scheme{PP\;3/4\;A\;3\;c}, \textrm{tol}$=10^{-5}$ }    & $408$ & $870$ & $83.59$ & $88.56$ \\ \hline
%{ \scheme{PP\;3/4\;A\;3\;c}, \textrm{tol}$=10^{-8}$ }    & $15050$ & $59947$ & $3190.13$ & $6477.74$ \\
%\hline
%\end{tabular}
%\caption{Efficiency of the \scheme{PP\;3/4\;A\;3\;c} splitting
%for~\eqref{gierer}. The tolerances were $10^{-5}$ (top) and $10^{-8}$ (bottom), respectively.\label{efficiencygmabc}}
%\end{center}
%\end{table}
%}

\section{Conclusions and outlook}

We have investigated high-order adaptive time-splitting methods for the solution of
nonlinear evolution equations of parabolic type under periodic
boundary conditions. The theoretical error analysis
\begin{rev}for the Gray--Scott equations and the Van der Pol equation\end{rev} shows
the classical convergence orders under regularity assumptions on the exact solution
implied by the Sobolev inequality for functions on the torus. The theory is illustrated
by numerical computations \begin{rev}showing the established convergence orders.\end{rev}

Moreover, adaptive time-stepping strategies \begin{rev}have been\end{rev} demonstrated
to improve both efficiency and reliability, where high-order methods generally yield a
computational advantage for the approximation of regular solutions. Local error estimators
based on embedded formulae of splitting coefficients are more efficient than estimators
\begin{rev}employing the adjoint method\end{rev}, but the former need
to be constructed especially by a computationally demanding optimization procedure,
while the latter principle can be applied invariantly for methods of odd order.

\begin{rev}Indeed, it has been observed that for problems with rapidly varying solutions,
an adaptive strategy yields an advantage as compared to uniformly using the
smallest time-step required locally. Secondly, a good guess of the time
step-size is not commonly available even when the solution is smooth, so adaptive adjustment saves from repeating
runs until the optimal step-size is found.
\end{rev}

Splitting into three operators promises a computational advantage for the calculation
of the individual compositions, but the complexity of high-order integrators of
this class implies a significant surplus of necessary compositions which negatively
affects the performance.

\appendix

\section{Periodic functions and their Fourier transforms} \label{sec:period}

%%%%%%%%%%%%%%%%%%%%%%%%%%%%%%%%%%%%%%%%%%%%%%%%%%%%%%%%%%%%%%%%%%
%%%%%%%%%%%%%%%%%%%%%%%%%%%%%%%%%%%%%%%%%%%%%%%%%%%%%%%%%%%%%%%%%%
In the following, we recapitulate material from~\cite{robinson01}
for the convenience of the reader. Consider
\begin{align*}
%\label{eq:Q}
Q &=[-a,a]^d \quad \text{associated with the $ d $\,-\,dimensional torus in $ \CC^d $}, \nl
C^n &= \{u\!:\, Q \to \CC, ~~ u \in C^n(Q) ~\text{is a periodic function} \}.
\end{align*}
The space $ L^2 = L^2(Q) $ is a Hilbert space with the inner product
\begin{equation*}
%\label{eq:L2}
{\il u,v \ir}_{L^2} = \int_Q \overline{u(x)}\,v(x)\,\dd x.
\end{equation*}
\paragraph{Fourier representation of $ u \in L^2 $}
Let $ k = (k_1,\ldots,k_d) \in \ZZ^d $, and $ |k| = |k_1| + \cdots + |k_d| $.
\begin{definition} %\label{def:fhat}
The Fourier transform
$\pazocal{F}: L^2(Q)\to {\mathbb Z}^d,\ u \mapsto \pazocal{F}(u) = (c_k)_{k\in \mathbb Z}$ is defined by
\begin{equation*}
c_k := \frac{1}{{(2\,a)}^d} \int_Q u(x)\,\ee^{-\ii\,(k \d x)/a}\,\dd x, \quad k \in \ZZ^d,
\end{equation*}
and the inverse transform yields the representation
\end{definition}
\begin{equation*}
%\label{eq:u-fourier}
u(x) = \sum_{k \in \ZZ^d}\,c_k\,\ee^{\ii\,\pi\,(k \d x) / a}.
\end{equation*}
Parseval's identity implies an isometric correspondence
\begin{equation}
\label{eq:parseval}
{\| u \|}_{L^2} =
{\Big( {\big( 2\,a \big)}^d \sum_{k \in \ZZ^d}\,|c_k|^2 \Big)}^{\frac{1}{2}}.
\end{equation}
\begin{remark}
Since the torus has finite measure, we have $ L^q \subseteq L^p $ for $ 1 \leq p \leq q \leq \infty $.
\end{remark}
We introduce the following notations:
$ H^s = H^s(Q) $,
$ \alpha = (\alpha_1,\ldots,\alpha_d) \in \NN_0^d $,\, $ |\alpha| = \alpha_1 + \cdots + \alpha_d $,\,
$ \alpha! = \alpha_1!\,\cdots\,\alpha_d! $. Weak derivatives are denoted by $ D^\alpha u $.
The norm on $ H^s $ is
\begin{subequations}
\begin{equation*}
%\label{eq:Hsnorm}
{\| u \|}_{H^s} = {\Big( \sum_{|\alpha| \leq s} {|D^\alpha u|}^2 \Big)}^{\frac{1}{2}}.
\end{equation*}
$ H^s $ is a Hilbert space with inner product
\begin{equation*}
%\label{eq:Hsip}
{\il u,v \ir}_{H^s} = \sum_{|\alpha| \leq s} \il D^\alpha u,D^\alpha v \ir_{L^2}.
\end{equation*}
\end{subequations}

\paragraph{Fourier representation of $ D^\alpha u $.}
The weak derivative has the Fourier representation
\begin{subequations}
%\label{eq:Dalphau-fourier}
\begin{align*}
%\label{eq:Dalphau-fourier-1}
D^\alpha u(x)
&= {\Big( \frac{\ii\,\pi}{a} \Big)}^{|\alpha|}
   \sum_{k \in \ZZ^d} k^\alpha\,c_k\,\ee^{\ii\,\pi\,(k \d x) / a}, \nl
%%\label{eq:Dalphau-fourier-2}
%{(D^\alpha u(x))}^{-}
%&  = {\Big( \frac{-\ii\,\pi}{a} \Big)}^{|\alpha|}
%   \sum_{k \in \ZZ^d} k^\alpha\,{\bar c}_k\,\ee^{-\ii\,\pi\,(k \d x) / a},
\end{align*}
and thus
\begin{equation*}
%\label{eq:Dalphau-fourier-3}
{|D^\alpha u(x)|}^2
 = {\big( 2\,a \big)}^d\,{\Big( \frac{\pi}{a} \Big)}^{2|\alpha|}
                          \sum_{k \in \ZZ^d} k^{2\alpha}\,{|c_k|}^2
\end{equation*}
\end{subequations}
as a consequence of Parseval's identity~\eqref{eq:parseval}.
Here, $ k^\alpha = k_1^{\alpha_1} \cdots k_d^{\alpha_d} $.

In the following, we will need to resort to the fact that the norms on the Sobolev space
$H^s$ can equivalently be stated in terms of the Fourier coefficients. The proof
of the following lemma is given in~\cite{robinson01}.

\begin{lemma} \label{lemma:HsC-fourier}
With computable constants $ \underline{C},\,\overline{C} $ depending on $ d $ and $ s $ we have
\begin{equation*}
%\label{eq:HsC-fourier}
\underline{C}\,{\| u \|}_{H^s}
\leq {\Big( {\big( 2\,a \big)}^d \sum_{k \in \ZZ^d} \big( 1 + {|k|}^{2s} \big)\,|c_k|^2 \Big)}^{\frac{1}{2}} \leq
\overline{C}\,{\| u \|}_{H^s}.
\end{equation*}
\end{lemma}
Lemma~\ref{lemma:HsC-fourier} shows that $ H^s $ is identical to the space
\begin{subequations}
\label{eq:Hsnorm-fourier}
\begin{equation}
\label{eq:Hsnorm-fourier-1}
\Big\{ u(x) = \sum_{k \in \ZZ^d}\,c_k\,\ee^{\ii\,\pi\,(k \d x) / a} \in L^2,
      ~~\sum_{k \in \ZZ^d}\,\big( 1+{|k|}^{2s} \big)\,{|c_k|}^2 < \infty \Big\},
\end{equation}
and the norm $ {\| u \|}_{H^s} $ is equivalent to the norm
\begin{equation}
\label{eq:Hsnorm-fourier-2}
{\| u \|}_{H_\ast^s} =
{\Big( {\big( 2\,a \big)}^d \sum_{k \in \ZZ^d} \big( 1 + {|k|}^{2s} \big)\,|c_k|^2 \Big)}^{\frac{1}{2}}.
\end{equation}
\end{subequations}
Moreover, \eqref{eq:Hsnorm-fourier} serves as the definition of the spaces $ H^s $ for
non-integer $ s $.
%
%%%%%%%%%%%%%%%%%%%%%%%%%%%%%%%%%%%%%%%%%%%%%%%%%%%%%%%%%%%%%%%%%%
%%%%%%%%%%%%%%%%%%%%%%%%%%%%%%%%%%%%%%%%%%%%%%%%%%%%%%%%%%%%%%%%%%
\section{Sobolev embeddings} \label{sec:sobemb}
%%%%%%%%%%%%%%%%%%%%%%%%%%%%%%%%%%%%%%%%%%%%%%%%%%%%%%%%%%%%%%%%%%
%%%%%%%%%%%%%%%%%%%%%%%%%%%%%%%%%%%%%%%%%%%%%%%%%%%%%%%%%%%%%%%%%%
%%%%%%%%%%%%%%%%%%%%%%%%%%%%%%%%%%%%%%%%%%%%%%%%%%%%%%%%%%%%%%%%%%
%%%%%%%%%%%%%%%%%%%%%%%%%%%%%%%%%%%%%%%%%%%%%%%%%%%%%%%%%%%%%%%%%%
\subsection{Continuity} %\label{sec:HsC}
%%%%%%%%%%%%%%%%%%%%%%%%%%%%%%%%%%%%%%%%%%%%%%%%%%%%%%%%%%%%%%%%%%
%%%%%%%%%%%%%%%%%%%%%%%%%%%%%%%%%%%%%%%%%%%%%%%%%%%%%%%%%%%%%%%%%%
\begin{theorem}
\label{theorem:HsC}
For $ s > d/2 $ we have $ H^s \subseteq C^0 $, and the embedding
$ H^s \hookrightarrow C^0 $ is continuous, i.e.,
\begin{equation}
\label{eq:HsC}
{\| u \|}_{\infty} \leq \pazocal{C}_s\,{\| u \|}_{H^s} \quad \text{for all}~~ u \in H^s.
\end{equation}
\end{theorem}
%\begin{remark}
%This coincides with the assertion for the case $ H^s(\RR^d) $.
%This is not surprising since differentiability is a local property.
%\end{remark}
\begin{proof}
The proof is indicated in~\cite{robinson01}. In the following we work out the
argument in detail.
Consider an arbitrary $ u \in H^s $. With
\begin{align*}
|u(x)| &= \Big| \sum_{k \in \ZZ^d}\,c_k\,\ee^{\ii\,\pi\,(k \d x) / a}\,\Big|
       \leq \sum_{k \in \ZZ^d}\,|c_k|, \\
{\| u \|}_\infty &\leq \sum_{k \in \ZZ^d}\,|c_k|,
\end{align*}
the Cauchy-Schwarz inequality in $ \ell^2 = \ell_d^2 $ yields
\begin{align*}
{\| u \|}_\infty
&\leq \sum_{k \in \ZZ^d}\,\frac{1}{{\big( 1 + {|k|}^{2s} \big)}^{\frac{1}{2}}}\,
                          {\big( 1 + {|k|}^{2s} \big)}^{\frac{1}{2}}\,|c_k| \\
&\leq \bigg( \sum_{k \in \ZZ^d}\,\frac{1}{1 + {|k|}^{2s}} \bigg)^{\frac{1}{2}} \cdot
      \underbrace{\bigg( \sum_{k \in \ZZ^d}\,\big( 1 + {|k|}^{2s} \big)\,{|c_k|}^2 \bigg)^{\frac{1}{2}}}_{=\,C\,{\| u \|}_{H_\ast^s}}
\end{align*}
(with $ C = {\big( 2\,a \big)}^{-\frac{d}{2}} $), provided that the series
\begin{equation}
\label{eq:k2s-series}
\sum_{k \in \ZZ^d}\,\frac{1}{1 + {|k|}^{2s}}
\end{equation}
is convergent.
\begin{itemize}
\item For $ d=1 $,
\begin{equation*}
\sum_{k_1=-\infty}^{\infty}\,\frac{1}{1 + {|k_1|}^{2s}}
= 1 + 2 \sum_{k_1=1}^{\infty}\,\frac{1}{1 + {|k_1|}^{2s}}\,,
\end{equation*}
where
\begin{equation*}
\sum_{k_1=1}^{\infty}\,\frac{1}{1 + {|k_1|}^{2s}} \leq
\sum_{k_1=1}^{\infty}\,\frac{1}{{|k_1|}^{2s}}
\end{equation*}
is convergent for $ 2s>1 $, i.e., $ s > 1/2 = d/2 $.
\item For general $ d $ we consider
\begin{align*}
\sum_{k \in \ZZ^d}\,\frac{1}{1 + {|k|}^{2s}}
&\leq C
\sum_{k \in \NN_0^d}\,\frac{1}{1 + {|k|}^{2s}}
= \sum_{m=0}^{\infty}\,\sum_{\stackrel{|k|=m}{k \in \NN_0^d}}\,\frac{1}{1 + m^{2s}} \\
&= \sum_{m=0}^{\infty}\,\binom{m+d-1}{d-1}\,\frac{1}{1 + m^{2s}}
   \leq C \sum_{m=0}^{\infty}\,\frac{m^d}{1 + m^{2s}} < \infty
\end{align*}
for $ s>d/2 $.
\end{itemize}
This shows that, for $ s > d/2 $, the series~\eqref{eq:k2s-series} is convergent and that
$ u \in H^s $ satisfies~\eqref{eq:HsC}. Furthermore, the absolute summability of the
Fourier coefficients $ c_k $ implies that the Fourier series for $ u $ is uniformly
convergent, which in turn implies the continuity of $ u $.
\qed
\end{proof}

\begin{corollary}
%\label{corollary:HsCp}
For $ s > d/2 + n $ we have $ H^s \subseteq C^n $, and the embedding
$ H^s \hookrightarrow C^n $ is continuous, i.e.,
\begin{equation*}
%\label{eq:HsCp}
{\| u \|}_{C^n} \leq \pazocal{C}_{s,n}\,{\| u \|}_{H^s} \quad \text{for all}~~ u \in H^s.
\end{equation*}
\end{corollary}
%%%%%%%%%%%%%%%%%%%%%%%%%%%%%%%%%%%%%%%%%%%%%%%%%%%%%%%%%%%%%%%%%%
%%%%%%%%%%%%%%%%%%%%%%%%%%%%%%%%%%%%%%%%%%%%%%%%%%%%%%%%%%%%%%%%%%
\subsection{Integrability} %\label{sec:HsL}
%%%%%%%%%%%%%%%%%%%%%%%%%%%%%%%%%%%%%%%%%%%%%%%%%%%%%%%%%%%%%%%%%%
%%%%%%%%%%%%%%%%%%%%%%%%%%%%%%%%%%%%%%%%%%%%%%%%%%%%%%%%%%%%%%%%%%
In order to study integrability properties of functions $ u \in H^s $ we need to interrelate
them to summability properties of its Fourier transform in $ \ell^q $ spaces, with
\begin{equation*}
%\label{eq:lqnorm}
{\| {\hat u} \|}_{\ell^q} = {\Big( \sum_{k \in \ZZ^d} {|c_k|}^q \Big)}^{\frac{1}{q}}
\end{equation*}
($ {\hat u} = ( u_k )_{k \in \ZZ^d} $).
For the proof of the following result see~\cite[Theorem 2.1 \& 2.2]{katz68} and also
\cite{rudin87}.

\begin{lemma}[Hausdorff-Young]
%\label{lemma:rudin}
Let $ 1 \leq p \leq 2 $ and $ \frac{1}{p} + \frac{1}{q} = 1 $. Then
\begin{subequations}
\label{wq:HY}
\begin{equation}
\label{eq:HY1}
{\| {\hat u} \|}_{\ell^p} \leq C\,{\| u \|}_{L^q},
\end{equation}
and
\begin{equation}
\label{eq:HY2}
{\| u \|}_{L^q} \leq C\,{\| {\hat u} \|}_{\ell^p}.
\end{equation}
\end{subequations}
\end{lemma}
%\begin{proof}
%See~\cite[Theorem 2.1 \& 2.2]{katz68} and also
%see also~\cite{rudin87}.
%Here we only consider
%two extremal cases concerning~\eqref{eq:HY2}:
%\begin{equation*}
%{\| u \|}_{L^2} = {\| {\hat u} \|}_{\ell^2}
%\end{equation*}
%corresponds with Parseval's identity, and
%\begin{equation*}
%{\| u \|}_{L^\infty} \leq C\,{\| {\hat u} \|}_{\ell^1}
%\end{equation*}
%follows from
%\begin{equation*}
%\int_Q \Big( \sum_{k \in \ZZ^d}\,c_k\,\ee^{\ii\,\pi\,(k \d x) / a} \Big)\,dx
%= \sum_{k \in \ZZ^d}\,c_k\,\int_Q \ee^{\ii\,\pi\,(k \d x) / a}\,dx.
%\end{equation*}
%\qed
%\end{proof}

\begin{theorem}
\label{theorem:HsL}
For $ s < d/2 $ and
\begin{subequations}
\begin{equation*}
%\label{eq:HsLp}
2 \leq p < \frac{d}{\frac{d}{2}-s}
\end{equation*}
we have $ H^s \subseteq L^p $, and the embedding
$ H^s \hookrightarrow L^p $ is continuous, i.e.,
\begin{equation*}
%\label{eq:HsL}
{\| u \|}_{L^p} \leq C\,{\| u \|}_{H^s} \quad \text{for all}~~ u \in H^s.
\end{equation*}
\end{subequations}
\end{theorem}
\begin{remark}
It can be shown that the assertion of Theorem~\ref{theorem:HsL}
is also valid for the endpoint case $ s = d/2 $ and $ p<\infty $; see~\cite{robinson01}.
% details:~\cite[Corollary 1.2]{benoh03}.
\end{remark}
\begin{proof}
The proof is indicated in~\cite{robinson01}.
%The proof is indicated in~\cite{benoh03}.
In the following we work out the argument in detail.

For $ p=2 $ the assertion is trivial. For $ 2 < p < \infty $ and $ \frac{1}{p} + \frac{1}{q} = 1 $,
inequality~\eqref{eq:HY2}
implies\footnote{Here, $ p $ plays the role of $ q $ in~\eqref{eq:HY2}
                 and vice versa. We have $ 1 \leq q \leq 2 $.}
\begin{align*}
{\| u \|}_{L^p}
&\leq C\,{\| {\hat u} \|}_{\ell^q}
= C\, {\Big( \sum_{k \in \ZZ^d} {|c_k|}^q \Big)}^{\frac{1}{q}}  \\
&= C\, {\Bigg( \sum_{k \in \ZZ^d}
         \bigg({\big( 1+{|k|}^{2s} \big) {|c_k|}^2 \bigg)}^{\frac{q}{2}}\,
         {\bigg( 1+{|k|}^{2s} \bigg)}^{-\frac{q}{2}} \Bigg)}^{\frac{1}{q}} \\
&\leq C\,\bigg[\,{\Big( \sum_{k \in \ZZ^d} {\big( 1+{|k|}^{2s} \big)} {|c_k|}^2 \Big)}^{\frac{q}{2}}\,
                 {\Big( \sum_{k \in \ZZ^d} {\big( 1+{|k|}^{2s} \big)}^{-\frac{q}{2-q}} \Big)}^{\frac{2-q}{2}}\,
           \bigg]^{\frac{1}{q}} \\
&= C\,{\Big( \sum_{k \in \ZZ^d} {\big( 1+{|k|}^{2s} \big)} {|c_k|}^2 \Big)}^{\frac{1}{2}}\,
      {\Big( \sum_{k \in \ZZ^d} {\big( 1+{|k|}^{2s} \big)}^{-\frac{q}{2-q}} \Big)}^{\frac{2-q}{2q}} \\
&\leq C\,{\| u \|}_{H^s}\,
     {\Big( \sum_{k \in \ZZ^d} {\big( 1+{|k|}^{2s} \big)}^{-\frac{q}{2-q}} \Big)}^{\frac{2-q}{2q}}.
\end{align*}
Here we have used H{\"o}lder's inequality
with conjugate exponents $ \frac{2}{q},\frac{2}{2-q} $,
and Lemma~\ref{lemma:HsC-fourier}. This estimate makes sense provided the sum in the
latter expression is finite, i.e., if
\begin{equation*}
\sum_{k \in \ZZ^d}\,{\bigg( \frac{1}{1 + {|k|}^{2s}} \bigg)}^{\frac{q}{2-q}}
{=} C
\sum_{k \in \NN_0^d}^{\infty}\,{\bigg( \frac{1}{1 + {|k|}^{2s}} \bigg)}^{\frac{q}{2-q}} < \infty\,.
\end{equation*}
We reason as in the proof of Theorem~\ref{theorem:HsC}: We have
\begin{align*}
&\sum_{k \in \NN_0^d}\,{\bigg( \frac{1}{1 + {|k|}^{2s}} \bigg)}^{\frac{q}{2-q}}
= \sum_{m=0}^{\infty}\,\sum_{\stackrel{|k|=m}{k \in \NN_0^d}}
   {\bigg( \frac{1}{1 + m^{2s}} \bigg)}^{\frac{q}{2-q}} \\
&~= \sum_{m=0}^{\infty}\,\binom{m+d-1}{d-1}
   {\bigg( \frac{1}{1 + m^{2s}} \bigg)}^{\frac{q}{2-q}}
   \leq C \sum_{m=0}^{\infty}\,
   \frac{m^d}{{\big( 1 + m^{2s} \big)}^{\frac{2sq}{2-q}}} < \infty
\end{align*}
for $ (2sq)/(2-q) > d $, i.e., $ q > 2d/(2s+d) $. With $ 1/p + 1/q = 1 $ this is
equivalent to $ p < d/(\frac{d}{2}-s) $, as asserted.
\qed
\end{proof}

In the special cases $ d=1,2,3 $, which are relevant to our analysis, this means:
\begin{subequations}
\begin{itemize}
\item $ d=1 $: For $ s < 1/2  $ and
\begin{equation*}
2 \leq p < \frac{1}{\frac{1}{2}-s}
\end{equation*}
we have $ H^s \subseteq L^p $.
\item $ d=2 $: For $ s < 1  $ and
\begin{equation*}
2 \leq p < \frac{2}{1-s}
\end{equation*}
we have $ H^s \subseteq L^p $. In particular,
$ H^1 \subseteq L^p $ for all $ 1 \leq p < \infty $.
\item $ d=3 $: For $ s < 3/2  $ and
\begin{equation*}
2 \leq p < \frac{3}{\frac{3}{2}-s}
\end{equation*}
we have $ H^s \subseteq L^p $. In particular, $ H^1 \subseteq L^6 $.
\end{itemize}
\end{subequations}
%%%%%%%%%%%%%%%%%%%%%%%%%%%%%%%%%%%%%%%%%%%%%%%%%%%%%%%%%%%%%%%%%%
%%%%%%%%%%%%%%%%%%%%%%%%%%%%%%%%%%%%%%%%%%%%%%%%%%%%%%%%%%%%%%%%%%

\end{document}